\documentclass[10pt]{amsart}%
\usepackage{amsmath}
\usepackage{amsfonts}
\usepackage{amsthm}
\usepackage{amssymb}
\usepackage{graphicx}%
\newtheorem{theorem}{Theorem}
\newtheorem*{theorem*}{Theorem}
\newtheorem*{acknowledgement*}{Acknowledgement}
\newtheorem*{definition*}{Definition}

\newtheorem{corollary}[theorem]{Corollary}

\newtheorem{lemma}[theorem]{Lemma}

\newtheorem{proposition}[theorem]{Proposition}
\newtheorem{remark}[theorem]{Remark}

\newcommand{\RR}[1]{{\mathbb{R}}^{#1}}
\newcommand{\pd}[2]{\frac{\partial #1}{\partial#2}}

\newcommand{\pdt}[0]{\frac{\partial}{\partial t}}

\newcommand{\mc}[1]{{\mathcal{#1}}}
\newcommand{\gt}[0]{\tilde{g}}

\newcommand{\delb}[0]{\overline{\nabla}}

\newcommand{\Gamb}[0]{\overline{\Gamma}}

\newcommand{\dt}[1]{\frac{\partial^{#1}}{\partial t^{ #1}}}

\newcommand{\Rc}[0]{\operatorname{Rc}}
\newcommand{\Rm}[0]{\operatorname{Rm}}

\newcommand{\R}[2]{R^{(#1, #2)}}
\newcommand{\A}[2]{A^{(#1, #2)}}
\newcommand{\T}[2]{T^{(#1, #2)}}
\newcommand{\h}[2]{h^{(#1, #2)}}
\newcommand{\g}[2]{g^{(#1, #2)}}

\newcommand{\dfn}[0]{\doteqdot}
\newcommand{\lt}[0]{\preceq}
\newcommand{\NN}[0]{\mathbb{N}}
\newcommand{\ZZ}[0]{\mathbb{Z}}

\title[Time-analyticity of the Ricci flow]{Time-analyticity of solutions to the Ricci flow}

\author{Brett Kotschwar}
\address{Arizona State University, Tempe, AZ, USA}
\email{kotschwar@asu.edu}

\thanks{The author was supported in part by NSF grant DMS-1160613.}

\date{March 2012}

\begin{document}
 \begin{abstract}
  In this paper, we prove that if $g(t)$ is a smooth, complete solution to the Ricci flow of uniformly bounded curvature on 
$M\times[0, \Omega]$,  
 then the correspondence $t\mapsto g(t)$ is real-analytic at each $t_0\in (0, \Omega)$.  The analyticity
is a consequence of classical Bernstein-type estimates on the temporal and spatial derivatives of the curvature tensor,
which we further use to show that, under the above global hypotheses, for any $x_0\in M$ and 
$t_0\in (0, \Omega)$, there exist local coordinates $x = x^i$ on a neighborhood $U\subset M$ of $x_0$ in which the representation
$g_{ij}(x, t)$ of the metric is real-analytic in both $x$ and $t$ on some cylinder $U\times (t_0 -\epsilon, t_0 + \epsilon)$.
\end{abstract}
\maketitle

\section{Introduction.} 
Let $M$ be a smooth $n$-dimensional manifold and $g_0$ a Riemannian metric on $M$.  We will consider solutions
$g(t)$ to the Ricci flow
\begin{equation}\label{eq:rf}
  \pdt g = -2\Rc(g), \quad g(0)= g_0,
\end{equation}
on $M \times [0, \Omega]$. 

When $M$ is compact, it is an old result of Bando \cite{Bando} that $(M, g(t_0))$ is a real-analytic manifold for $t_0\in (0, \Omega]$.  His argument extends, essentially without change, 
to non-compact $M$, provided the solution $g(t)$ is complete and of uniformly bounded curvature (cf. \cite{RFV2P2}),
and, with some modification, to \emph{any} smooth solution (see \cite{KotschwarLocalAnalyticity}). 
This is not surprising, given the close analogy of the Ricci flow 
to the linear heat equation, whose solutions possess instantaneous real-analyticity in the spatial variables
 as a consequence of parabolic regularity.

Under suitable global assumptions, solutions to the heat equation will also possess instantaneous analyticity in time,
and it is a natural to ask whether solutions to the Ricci flow share this property.  In this paper we  prove
that they do so, at least in the class of complete solutions of uniformly bounded curvature. 
\begin{theorem}\label{thm:timeanalyticity}
  Suppose $(M, g_0)$ is complete and $g(t)$ is a smooth solution to 
  \eqref{eq:rf} satisfying $\sup_{M\times[0, \Omega]}|\Rm(x, t)| \leq M_0$.
   Then the map $g: (0, \Omega) \to X$ is real-analytic, where
   $X$ denotes the Banach space $BC(T_2(M))$ equipped with the supremum norm $\|\cdot\|_{g(0)}$ relative to $g(0)$.

\end{theorem}
We suspect that the general form of the theorem, namely, a conclusion that is \emph{interior} in time
from assumptions that are \emph{global} in space and time, 
is probably as general as might be hoped, as local analyticity in space and time is false even for the linear heat equation.  
We do not know whether the bound on curvature is optimal, although, with the assumption
of completeness, it essentially describes  the most general class of solutions for which short-time existence  and 
forwards and backwards uniqueness are known to hold (cf. \cite{Hamilton3D}, \cite{Shi}, \cite{ChenZhu}, \cite{KotschwarBackwardsUniqueness}, \cite{KotschwarUniqueness}).
This theorem provides an alternative and unified proof of the following unique-continuation
property of solutions to the Ricci flow, which follows from the combined uniqueness 
results in \cite{Hamilton3D}, \cite{ChenZhu}, and \cite{KotschwarBackwardsUniqueness}. 

\begin{corollary}\label{cor:uc}
 Suppose $g(t)$ and $\gt(t)$ are smooth, complete solutions to \eqref{eq:rf} of uniformly bounded curvature 
on $M\times[0, \Omega]$. If $g(t_0) \equiv \gt(t_0)$ for some $t_0\in (0, \Omega)$,
then $g(t) \equiv \gt(t)$ for all $t\in [0, \Omega]$.
\end{corollary}

The question of time-analyticity for solutions to parabolic equations is an old and well-studied problem,
with many deep and general results, typically attained by way of either semigroup methods
(see, e.g., \cite{Yosida}, \cite{Komatsu2}, \cite{Massey}, \cite{Lunardi}) or $L^2$-estimates
(see, e.g., \cite{KinderlehrerNirenberg}, \cite{Komatsu}). However, as the Ricci flow is not strictly parabolic, the existing
theory provides no automatic guarantee of the time-analyticity of its solutions. Moreover, although, relative to choice of 
a fixed background metric $\bar{g}$ (or connection), one can associate to a solution $g(t)$ 
of the Ricci flow a solution $h_{\bar{g}}(t)$ to the strictly parabolic
Ricci-DeTurck flow \cite{DeTurck} for which $g(t) = \phi^{*}_th_{\bar{g}}(t)$ for some smooth family $\phi_t \in \operatorname{Diff}(M)$,  
it appears to be somewhat problematic, at the very least, to parlay 
any statement of analyticity for $h_{\bar{g}}(t)$ into a \emph{tensorial}
statement of analyticity for $g(t)$ of the sort in Theorem \ref{thm:timeanalyticity}.
Our approach, much the same as in \cite{Bando}, is to establish
 the time-analyticity of $g(t)$ by way of direct and rather classical Bernstein-type estimates on the curvature tensor.
The nature of our proof requires us to establish bounds
on all mixed covariant and temporal derivatives of the curvature tensor, and thus we obtain the following
estimates which are stronger than what is necessary to imply the time-analyticity of $g(t)$ but may  be themselves of some independent interest.
\begin{theorem}\label{thm:mainest}
  Given $M_0$ and $\Omega > 0$, there exist constants $K_0$ and $L_0$ depending
  only on $n$, $M_0$, and $\Omega^* \dfn \max\{\Omega, 1\}$, such that the curvature tensor
  of any solution $g(t)$ to \eqref{eq:rf} satisfying the assumptions of
  Theorem \ref{thm:timeanalyticity} satisfies 
 \begin{equation}\label{eq:mainest}
   \sup_{M\times [0, \Omega]} 
    t^{k/2 + l}\left|\nabla^{(k)}\dt{l}\Rm\right| \leq K_0L_0^{k/2 + l}(k+l)!
  \end{equation}
for all $k$, $l\in \NN_0 \dfn \NN\cup\{0\}$.
\end{theorem}

Above, and in what follows, $\nabla \dfn \nabla_{g(t)}$ denotes the Levi-Civita connection of $g(t)$,
$|\cdot| = |\cdot|_{g(t)}$ the norms induced by $g(t)$ on the tensor bundles $T^k_l(M)$, and
$\nabla^{(k)}T$ the $k$-fold covariant derivative of a tensor $T$. We reserve the designation
``the curvature tensor of $g(t)$'' for the $(3, 1)$-curvature tensor, i.e., $R_{ijk}^l$,
which we will often simply denote by $R$. Since $\pdt g_{ij} = -2R^{l}_{lij}$, the above estimates
immediately imply estimates of the same general form for the derivatives of $g(t)$; in view of the assumption of bounded curvature,
the metrics $g(t)$, $t\in [0,\Omega]$, will all be uniformly equivalent, so, at the expense of further enlarging the constants, 
we can replace the norms in these estimates
by those induced by $g(0)$.  The time-analyticity of $g(t)$ for $t > 0$, in the sense of Theorem \ref{thm:timeanalyticity}, then follows. 

\subsection{Remarks on the estimates}

The estimates of Theorem \ref{thm:mainest} generalize the ubiquitous spatial derivative estimates
due originally to Bando \cite{Bando} and Shi \cite{Shi},  
which take the form
\[
      t^{k/2}|\nabla^{(k)} R|\leq C(n, k, M_0, \Omega).
\]
Other variants appear, for example, in \cite{HamiltonSingularities},  \cite{LuTian}, \cite{ShermanWeinkove}, 
and \cite{KotschwarLocalAnalyticity}. The essential
new feature in \eqref{eq:mainest} is the precise dependency of the factorial on the
right-hand side of the equation on the the order, $l$, of the time-derivative. From \cite{Bando}
it is known that the constant 
$C$ can be put into the form $C^{\prime}L^k k!$ for some constants $C^{\prime}$ and $L$ 
depending only on $n$, $M_0$, and $\Omega$.  While
this rate of growth in the order $k$ is sufficient to establish the spatial analyticity of solutions for $t > 0$, any combination
of estimates of this form could, at best, imply an estimate of the form
\[
    t^{l}\left|\frac{\partial^lR}{\partial t^l}\right|\leq CL^{l}(2l)!
\]
on the time-derivatives of $R$. This dependency on $l$ implies that $g(t)$ is of second Gevrey class for $t > 0$, 
but is insufficient to ensure that it is analytic in $t$.

To improve the dependency of the estimates on the order of the time-derivative,
we must first make some adjustments to the Bernstein technique in \cite{Bando};
these high-level alterations are contained in Section \ref{sec:phiev}. Most of the rest
of the issues we need to address then arise more-or-less from the greater complexity of the computations we need to carry out.
 The evolution equation for $\nabla^{(k)}\dt{l}R$ effectively 
compels us to include all mixed derivatives $\nabla^{k^\prime}\dt{l^{\prime}}R$ with $k^{\prime}+l^{\prime}\leq k+l$ in our basic quantity, 
and, generally speaking, the nonlinearity of the equation proves to be somewhat more troublesome in this case than it is for the estimation
of purely spatial derivatives. For example, the repeated time-differentiation of expressions
 involving the inverse of the metric
increases not only the number of terms but the number and variety of factors in each term,
whereas (on account of the compatibility of the metric with the connection)
the iterated covariant differentiation of the same expressions typically only increases the number of terms. 

This threatens to make unmanageable some of the crucial commutator formulas (e.g., $\left[\dt{l}, \Delta\right]$)
that we must consider, but we are able to neutralize a large part of this potential complication by introducing an induction scheme
and treating the inverse independently of $g$ at each stage.  Then, from the identity
$g^{ik}g_{kj} = \delta^i_j$ we can deduce explicit bounds on the derivatives of the inverse
from estimates established on the derivatives of the metric.
This allows us to avoid having to work with explicit expressions for iterated derivatives of the inverse of the metric,
although, conceivably, one could avoid the induction with an application of a suitable variation of Fa\`{a} di Bruno's formula.

We note that, although we have made no effort to work in any generality greater than what is demanded by Theorem \ref{thm:timeanalyticity},
our argument makes no essential use of the Ricci flow equation, and particularly
as it concerns the basic quantity $\Phi_N$, defined in Section \ref{sec:indarg}, and the 
high-level estimates we obtain on its evolution equation in Section \ref{sec:phiev}, it illustrates a general
technique which can be used as well to establish the time analyticity of other parabolic equations.
For example, it can be used to prove the analyticity of solutions to the mean-curvature flow
in Euclidean space (or in 
in an ambient manifold with a curvature tensor 
satisfying appropriate bounds on its covariant derivatives); we intend to detail this particular extension in a paper to follow. 

\subsection{Full analyticity of the solution in local coordinates}
Since Theorem \ref{thm:mainest} provides estimates on the covariant as well as the temporal derivatives  of the curvature tensor,
it is natural to ask whether we might also use them to obtain some statement of local space-time analyticity for the metric.
In the last section of the paper we will we use them to prove that, for any $0 < t_0 < \Omega$,
 there exist coordinates $x$ about any $x_0\in M$ in which the expression $g_{ij}(x, t)$ of the metric
is analytic in $x$ \emph{and} $t$ in some small interior neighborhood of $(x_0, t_0)$. In fact, we will
show that \emph{any} local coordinates $x$ in which the representation $g_{ij}(x, t_0)$ at some $t_0 > 0$
is analytic in space will do the trick if restricted to a sufficiently small neighborhood.
The precise statement is as follows. 
\begin{theorem}\label{thm:spacetimeanalyticity} Let $g(t)$ be a complete solution to \eqref{eq:rf}
satisfying the uniform curvature bound $|\Rm(x, t)|\leq M_0$ on $M\times [0, \Omega]$.  Let $t_0 \in (0, \Omega)$ and suppose that
$x: U \to \RR{n}$ are coordinates
on a neighborhood $U\subset M$ in which the expression $x\mapsto g_{ij}(x, t_0)$ of the metric $g(t_0)$ belongs to $C^{\omega}(U)$.
Then, given any $x_0 \in U$, there exists a neighborhood $V\subset U$ of $x_0$ and $0 < \epsilon < \min\{t_0, \Omega - t_0\}$ 
such that $g_{ij}$ belongs to $C^{\omega}(V\times(t_0-\epsilon, t_0 + \epsilon))$. 
\end{theorem}
We emphasize again that, as in the case of strictly parabolic equations,
we require global assumptions on the solution $g(t)$ despite the purely local conclusion of the theorem. 
In view of Bando's theorem \cite{Bando} (cf. Remark 13.32 of \cite{RFV2P2}), we have the following special case.
\begin{corollary}\label{cor:spacetimeanalyticity} Suppose $g(t)$ is a complete solution to \eqref{eq:rf} of uniformly bounded curvature
on $M\times [0, \Omega]$.  Then, for each $(x_0, t_0)\in M\times (0, \Omega)$,
there exist $r > 0$ and $0 < \epsilon < \min\{t_0, \Omega-t_0\}$  such that the expression, $g_{ij}$, of the metric
in $g(t_0)$-geodesic normal coordinates $x = (x^i)$ based at $x_0$, is real-analytic in the variables $x^i$ and $t$
in the open cylinder $B_{g(t_0)}(x_0, r)\times (t_0 -\epsilon, t_0 + \epsilon)$. 
\end{corollary}
\begin{remark}
An analogous statement holds true for local $g(t_0)$-harmonic coordinates, 
as it is a result of DeTurck-Kazdan \cite{DeTurckKazdan} that a metric whose representation is of class $C^{\omega}$ 
in some coordinates $x$ will also be of class $C^{\omega}$ in harmonic coordinates.
\end{remark}

\section{Notation and conventions}
\subsection{Derivatives}
For the rather detailed calculations that lie ahead, 
it will be convenient to have a shorthand to represent iterated space and time derivatives.
We will use the notation
\begin{equation}\label{eq:derconvention}
	  \T{k}{l} \dfn \nabla^{(k)}\dt{l} T,
\end{equation}
for a smooth family $T =T(t)$ of tensors on $M$. Thus, in particular, if $T$ is a family of $(a, b)$-tensors on $M$, 
$\T{k}{l}$ represents a family of $(k+a, b)$-tensors. Since the connection $\nabla$ depends on $t$ through $g(t)$, 
the operators $\pdt$ and $\nabla$ typically will not commute. So $\dt{m}\T{k}{l}\neq \T{k}{l+m}$ in general, however,
trivially,
\[
      \nabla^{(m)}\T{k}{l} = \T{k+m}{l},\quad\mbox{and}\quad \dt{m}\T{0}{l} = \T{0}{l+m},
\]
for all $k$, $l$, $m\in \NN_0$.

\subsection{Contracted products of tensors}
As our calculations below are to be motivated by an almost exclusively combinatorial interest, 
it will be useful also to introduce a convention to 
suppress as much of the internal algebraic structure of a given formula
as it is safe to ignore. By
the expression
\begin{equation}\label{eq:ltdef}
      U \lt c V_1\ast V_2\ast \cdots \ast V_J,
\end{equation}
for $c\in \NN_0$ and tensors $U$ and $V_i$, $i=1,\ldots, J$, 
we will mean that $U$ is the sum of no more than $c$ terms of \emph{simple} (i.e., \emph{non-metric}) 
contractions of $V_1\otimes V_2\otimes \cdots \otimes V_J$, and we extend
this notation in the obvious way to sums of such terms.  In particular, an expression
of the form
\begin{equation*}
    U \preceq \sum_{i=1}^N c_i V_{i, 1}\ast V_{i, 2}\ast\cdots \ast V_{i, J_i}
\end{equation*}
where $c_i \geq 0$, implies that $U$ satisfies
\[
    |U|\leq \sum_{i=1}^N c_i |V_{i, 1}||V_{i, 2}|\cdots |V_{i, J_i}|.
\]
\begin{remark} Note that our use of the asterisk notation differs in one important aspect from its typical 
usage in the Ricci flow literature (e.g., as in \cite{Bando}, \cite{ChowKnopf}, \cite{HamiltonSingularities})
in that we do not use it to conceal any metric contractions.  Indeed, since we need to estimate 
temporal as well as covariant derivatives, it is virtually always necessary that we track 
each occurence of the metric and its inverse in the formulas. With this stipulation,
 the ``inequalities''  $U \lt c V_1\ast V_2\ast \cdots \ast V_J$ may be differentiated in both space
and time, that is,
if $U\lt c V\ast W$, then 
\[
\nabla U \lt c \nabla V\ast W + cV\ast \nabla W, \quad\mbox{and}\quad \pd{U}{t} \lt c\pd{V}{t}\ast W + cV\ast \pd{W}{t},
\] 
with obvious generalizations to products of greater numbers of factors.
\end{remark}
 
\subsection{Factorials and combinatorial conventions}
In our coefficients, we will use the convention that $m! \dfn 1$ for all $m\leq 0$ and the notation
\[
  \quad [m] \dfn \max\{m, 1\},\quad \mbox{and}\quad [m]_k \dfn \frac{m!}{(m-k)!} = [m][m-1]\cdots [m-k+1],
\]
for $m$, $k\in \ZZ$. We will further use the rightmost expression to interpret 
$[x]_r$ for $x\in \RR{}$, with $[x] = 1$ when $x\leq 0$ as in the integral case.

We will also use standard multi-index notation, for instance
\[
   |\alpha| \dfn \alpha_1 + \alpha_2 + \ldots + \alpha_l, \quad \binom{|\alpha|}{\alpha} \dfn
\frac{|\alpha|!}{\alpha_1!\alpha_2!\cdots! \alpha_l!},
\]
and 
\[
    [\alpha]_k \dfn [\alpha_1]_k\cdot [\alpha_2]_k\cdots [\alpha_l]_k, 
\]
for multi-indices $\alpha = (\alpha_1, \alpha_2, \ldots, \alpha_{l})$ of length $l$.
Finally, unless specified otherwise, the indices appearing in any set 
are assumed to belong to $\NN_0 \dfn \NN\cup\{0\}$, and any sum over an empty index set is to be interpreted as zero.

\section{Coefficients and combinatorial preliminaries}\label{sec:combin}

We aim to prove \eqref{eq:mainest} ultimately by an induction argument on the combined order, $k+l$, of the derivatives,
however, it turns out that it is much easier to obtain the estimates needed to complete the induction step
if we replace \eqref{eq:mainest} with the apparently stronger variant
\begin{equation}\label{eq:mainestvar}
    t^{k/2+l}|R^{(k, l)}| \leq KL^{k/2+l} (k-2)!(l-2)!.
\end{equation}
Equation \eqref{eq:mainest} (with $L_0 = 2L$) then follows readily
since the inequality $m! n! \leq (m+n)!$, which is valid for all $m$ and $n\geq 0$,
implies that, for $k$, $l\geq 2$,
\[
	 (k-2)!(l-2)! \leq(k+l -4)!,
\]
and, consequently, that
\[
      (k-2)!(l-2)!\leq 2^{k+l}(k + l - 4)!
\]
for all $k$, $l\in \NN_0$.  Of course, there is also
a universal constant $a$ such that
\begin{equation}\label{eq:reversefactineq}
      (m+n)! \leq a^{m+n}m!n!
\end{equation}
for all $m$ and $n\in \NN_0$, so there really is no essential difference in the strength of  
the estimates \eqref{eq:mainest} and \eqref{eq:mainestvar}.
The presence
of the delay (or ``lag'') in the factorials does, however, offer us an indispensable technical advantage
in that it allows us to infer like estimates on the derivatives of some product of tensors from
estimates on the derivatives of its factors.  We learned of this device of ``factorial lag'' 
in the paper of Kinderlehrer-Nirenberg \cite{KinderlehrerNirenberg}, who attribute it to Lax \cite{Lax}
(see also \cite{Friedman} and \cite{Komatsu}).  At its heart is the elementary observation that there exists a (universal) constant
$C$ such that
\[
  \sum_{p = 0}^m [p]_2^{-1}[m-p]_2^{-1} = \frac{1}{m!}\sum_{p=0}^m \binom{m}{p}(p-2)!(m-p-2)! \leq C\frac{(m-2)!}{m!} = C[m]_2^{-1}.
\]
In fact, we have the following more general lemma.
\begin{lemma}\label{lem:comb1} Suppose that $\beta$ is a multi-index of length $l\geq 2$.
If $\max_{1\leq i\leq l} \beta_{i} \geq 2$, then there exists a constant
$C\geq 1$ depending on $l$ and $r \dfn \min_{1\leq i\leq l} \beta_i$ (but not on $m$) such that
\[
  \sum_{|\alpha| = m} \prod_{i=1}^l [\alpha_i]_{\beta_{i}}^{-1} = \sum_{|\alpha|= m} \frac{(\alpha - \beta)!}{\alpha!} 
    \leq C \frac{(m-r)!}{m!} =  C [m]_{r}^{-1}.
\]
\end{lemma}
\begin{proof}
Since 
\[
  \sum_{|\alpha| = m} \prod_{i=1}^l [\alpha_i]_{\beta_{i}}^{-1}
  =\sum_{p=0}^m\left\{ \left(\sum_{|\alpha^{\prime}| = m-p}\left(\prod_{i=1}^{l-1}[\alpha^{\prime}_{i}]_{\beta_i}^{-1}\right)\right)\frac{1}{[p]_{\beta_l}} \right\},
\]
it suffices by induction to consider the case $l = 2$. 
By symmetry, we may write $\alpha = (p, m-p)$ and $\beta = (q, r)$ where $q \geq r$.  If $r = 0$, then the sum is simply
\[
 \sum_{p=0}^{m}\frac{1}{[p]_q}\leq \sum_{p=0}^{m}\frac{1}{[p]_2} \leq 2,
\]
and the claim holds trivially.   
If $r \geq 2$, since $q\geq r$, we  have
\begin{align*}
  &\sum_{p=0}^m [p]_{q}^{-1} [m-p]_r^{-1} \leq 2\sum_{0\leq p \leq m/2} [p]_{r}^{-1} [m-p]_r^{-1}\leq \frac{2}{[m/2]_r}\sum_{0\leq p \leq m/2} \frac{1}{[p]_r},
\end{align*}
and, since $[m-2a]/2\leq [m/2 - a]$ and $[m-a]/[m-2a] \leq 1+a$ for any $a\in \NN_0$, we have
\begin{align*}
  \frac{[m]_r}{[m/2]_r}&\leq 2^r\left(\frac{[m]}{[m]}\cdot \frac{[m-1]}{[m-2]}\cdot \frac{[m-2]}{[m-4]} \cdots \frac{[m-(r-1)]}{[m-2(r-1)]}\right)
		      \leq (2r)^r.
\end{align*}
Consequently,
\begin{align*}
  \sum_{p=0}^m [p]_{q}^{-1} [m-p]_r^{-1}
  &\leq \frac{2(2r)^r}{[m]_r}\sum_{p=0}^{\infty}\frac{1}{[p]_2} \leq \frac{6(2r)^{r}}{[m]_r}.
\end{align*}
Similarly, if $r = 1$,   we have, using that $q\geq 2$,
\begin{align*}
 \sum_{p = 0}^m \frac{1}{[p]_q [m-p]} 
  &\leq \sum_{0\leq p \leq m/2} \frac{1}{[p]_2 [m-p]} + \sum_{m/2\leq p \leq m} \frac{1}{[p]_2 [m-p]}\\
  &\leq\frac{2}{[m]}\left(\sum_{0\leq p \leq m/2} \frac{1}{[p]_2}\right) + \frac{2}{[m]}
    \left(\sum_{m/2\leq p \leq m} \frac{1}{[p-1][m-p]}\right)\\
    &\leq \frac{C}{[m]},
\end{align*}
for some universal constant $C$.
\end{proof}

Before we state the next lemma, let us introduce notation for the coefficients we will use as weights
in our estimates throughout the rest of the paper. Define 
\begin{equation}
      a_{k, l} \dfn \frac{1}{(k-2)!(l-2)!}.
\end{equation}
for $k$, $l\in \NN_0$.  Using Lemma \ref{lem:comb1}, we can show that when the norms of the elements of the family
 are weighted with these coefficients, bounds on the derivatives of families of tensors 
extend to estimates of their tensor products (and pure contractions thereof) with the same dependencies on the orders of the derivatives.

\begin{lemma}\label{lem:prodsum}
 Suppose that $U = U(t)$ and $V = V(t)$ are smooth families of tensors on $M$ for $t\in [0, \Omega]$ and
 that $W = W(t) = U\ast V$ is some simple contraction of their product.  Given $\theta > 0$, 
there exists a constant $C= C(n)$ such that, for any $N\in \NN_0$, if we write
       $u_{k, l} \dfn a^2_{k, l}\theta^{k+2l}|U^{(k, l)}|^2$ and $v_{k, l} \dfn a^2_{k, l}\theta^{k+2l}|V^{(k, l)}|^2$,
we have the estimate
\begin{equation}\label{eq:prodsum}
      \sum_{0\leq k+l \leq N} w_{k, l} \leq C\left(\sum_{0\leq k+l \leq N} u_{k, l}\right)
	\left(\sum_{0\leq k+l\leq N} v_{k, l}\right)
\end{equation}
on $w_{k, l} \dfn a^2_{k, l}\theta^{k+2l}|W^{(k, l)}|^2$ for $k$, $l\in \NN_0$.
\end{lemma}
\begin{proof} Throughout the proof, $C$ will denote a series of constants that depend only on $n$.
For any $k$ and $l\in \NN_{0}$, from the Leibniz rule, we have
\[
      |W^{(k, l)}| \leq  C\sum_{p=0}^k\sum_{q=0}^l\binom{k}{p}\binom{l}{q} |U^{(p, q)}||V^{(k-p, l-q)}|
\]
and so, by Cauchy-Schwarz,
\begin{align*}
      w_{k, l} &\leq C\left(\sum_{p=0}^k\sum_{q=0}^l\frac{\theta^{k/2+ l}[k]_2 [l]_2}{p!(k-p)!q!(l-q)!}|U^{(p, q)}||V^{(k-p, l-q)}|\right)^2\\
	       &\leq C\left(\sum_{p=0}^k\sum_{q=0}^l\frac{[k]_2 [l]_2}{[p]_2[k-p]_2[q]_2[l-q]_2}u_{p, q}^{1/2}v_{k-p, l-q}^{1/2}\right)^2\\
 	       &\leq C\left(\sum_{p=0}^k\sum_{q=0}^l\left(\frac{[k]_2 [l]_2}{[p]_2[k-p]_2[q]_2[l-q]_2}\right)^2\right)
 	         \left(\sum_{p=0}^k\sum_{q=0}^lu_{p, q}v_{k-p, l-q}\right).
\end{align*}
Now,
\[
    \sum_{p=0}^k\sum_{q=0}^l\left(\frac{[k]_2 [l]_2}{[p]_2[k-p]_2[q]_2[l-q]_2}\right)^2 \leq 
	  C\left(\sum_{p=0}^k \frac{[k]_4}{[p]_4[k-p]_4}\right)\left(\sum_{q=0}^l\frac{[l]_4}{[q]_{4}[l-q]_4}\right) \leq C,
\]
by Lemma \ref{lem:comb1}, so, for any $N\in \NN_0$,
\begin{align*}
    \sum_{0\leq k + l\leq N} w_{k, l} &\leq C \sum_{0\leq k + l \leq N}\left(\sum_{\stackrel{p+q=k}{r+s=l}}
    u_{p, r}v_{q, s}\right)\\
      &\leq C \left(\sum_{0\leq k+l\leq N} u_{k, l}\right) \left(\sum_{0 \leq k+l\leq N} v_{k, l}\right),
\end{align*}
for some $C = C(n)$ as claimed.
\end{proof}

\begin{remark} By induction, the above result can be extended to contractions of products of arbitrary length.
 In particular, \eqref{eq:prodsum} implies that if $V$ is a simple contraction of the $j$-fold tensor product of $U$ with itself, i.e.,
\[
      V = \underbrace{U\ast U \ast \cdots \ast U}_{j},
\]
then there exists a constant $C= C(n)$ such that, for any $N\in \NN_0$,
\[
     \sum_{0\leq k + l\leq N}\theta^{k+2l}a_{k, l}^2|V^{(k, l)}|^2 \leq C^{j} 
    \left(\sum_{0\leq k + l \leq N}\theta^{k+2l}a_{k, l}^2|U^{(k, l)}|^2\right)^{j}.
\]
\end{remark}

\section{The induction argument}\label{sec:indarg}

Now we specialize to the case of the Ricci flow, and henceforth suppose that $g(t)$ is a solution to \eqref{eq:rf} on 
$M\times [0, \Omega]$ satisfying the assumptions of Theorem \ref{thm:timeanalyticity}.
We will write $h(t) = g^{-1}(t)$ for the metric induced on $T^*(M)$ by $g(t)$, and let $L \geq 1$ denote a large positive constant
whose value we will eventually prescribe in terms of the external parameters $n$, $M_0$, and $\Omega^* = \max\{\Omega, 1\}$. 
Many of the expressions to follow will depend on $L$ through the function
\[
    \theta(t) \dfn\theta_L(t) \dfn \frac{t}{L},
\]
 but, except when the dependency needs particular emphasis, we will suppress the subscript $L$ in our notation.

Next, for any $k$, $l\in \NN_0$ (and fixed $L$), we define 
\begin{equation}\label{eq:phichidef}
   \phi_{k, l} \dfn \phi_{k, l; L} \dfn a^2_{k, l} \theta^{k+2l}|\R{k}{l}|^2,  \quad\mbox{and}\quad 
  \chi_{k, l} \dfn \chi_{k, l; L} \dfn a^2_{k, l}\theta^{k+2l}|\h{k}{l}|^2,
\end{equation}
and, for any $N\in \NN_0$,
\begin{equation}\label{eq:PhiChidef}
  \Phi_{N} \dfn \Phi_{N, L} \dfn \sum_{0\leq k + l\leq N}\phi_{k, l},\quad\mbox{and}\quad X_{N} \dfn X_{N; L}\dfn \sum_{0\leq k + l \leq N} \chi_{k, l}.
\end{equation}
Similarly, for $k \in \NN$ and $l\in \NN_0$, we define
\begin{equation}\label{eq:psidef}
    \psi_{k, l} \dfn \psi_{k, l; L} \dfn a^2_{k-1, l}\theta^{k +2l -1}|\R{k}{l}|^2,
\end{equation}
and, for any $N\in \NN$,
\begin{equation}\label{eq:Psidef}
 \Psi_N \dfn \Psi_{N; L} \dfn \sum_{\stackrel{1\leq k+l\leq N+1}{k\geq 1}}\psi_{k, l} = \sum_{m= 1}^{N+1}\sum_{k=1}^m \psi_{k, m-k}.
\end{equation}
Note that, for $t > 0$ and $k\in \NN$, we have
\begin{equation}\label{eq:psiphirel}
    \psi_{k, l}  = \frac{[k-2]^2}{\theta}\phi_{k, l}.
\end{equation}

Our goal will be to prove that, for an appropriate choice of the parameter $L$, $\Phi_{N, L}$ will 
be bounded on $M\times[0, \Omega]$ independently of $N$. Then \eqref{eq:mainest} will follow from
a comparison of the individual terms with this bound. The bound itself will be a 
consequence of iterating the following estimate to cover the interval $[0, \Omega]$.

\begin{theorem}\label{thm:phibound}
For any $M_0 > 0$ and $0 < \Omega_0 \leq 1$  there exist $L_1 = L_1(n, M_0) \geq 1$ and $\Omega_1 = \Omega_1(n, M_0)$ with $0 <\Omega_1\leq \Omega_0$, 
such that, if $g(t)$ is a solution to \eqref{eq:rf} on $M\times [0, \Omega_0]$ satisfying
\[
  \sup_{M\times [0, \Omega_0]}|R(x,t)|^2\leq M_0^2,
\]
then, for all $L \geq L_1$, 
\begin{equation}\label{eq:phibound}
\sup_{M\times [0, \Omega_1]} \Phi_{N; L}(x, t) \leq 2M_0^2
\end{equation}
for any $N\in \NN_0$.
\end{theorem}

We will prove Theorem \ref{thm:phibound} by an induction argument, the key to which is the following proposition.
Its proof will occupy the bulk of the sequel.
\begin{proposition}\label{prop:phiev}
  Given any $\Omega > 0$, and a complete solution to \eqref{eq:rf} with 
\[
    \sup_{M\times [0, \Omega]} |R(x,t)| \leq M_0,
\]
there exist constants $C_1 = C_1(n)$ 
  and $L_2 = L_2(n, M_0, \Omega^*)$ such that, whenever $L \geq L_2$
 and $\Phi_N = \Phi_{N; L}$ satisfies
\[
    \sup_{M\times [0, \Omega]}\Phi_{N}(x, t) \leq 2M_0^2
\]
for some $N \in \NN_0$,
then
   \begin{equation}\label{eq:phiev}
    \left(\pdt - \Delta\right)\Phi_{N+1} \leq -\left(1 - C_1\theta^2(\Phi_{N+1}^2+1)\right)\Psi_{N+1}
	+ C_1(M_0^2+1)(\Phi_{N+1}^2 + 1)
\end{equation}
on $M\times[0, \Omega]$.
\end{proposition}

\subsection{Proof of Theorem \ref{thm:phibound}}\label{sec:proofthmphibd}
We postpone the proof of Proposition \ref{prop:phiev} for the time-being and give first the argument
for Theorem \ref{thm:phibound}.

\begin{proof}[Proof of Theorem \ref{thm:phibound}, assuming Proposition \ref{prop:phiev}] We first claim
 that we may assume that, for all $L$ and $N$, we have a preliminary bound of the form
\[
  \sup_{M\times [0, \Omega_0]} \Phi_{N; L} \leq M_1(N, L). 
\]
To see this, note that, for any $0 < \epsilon <\Omega$, we can replace $g(t)$ with $g_{\epsilon}(t) = g(t+\epsilon)$
on $M\times [0, \Omega_\epsilon]$, where $\Omega_{\epsilon}\dfn \Omega -\epsilon$, and consider the corresponding expression
 $\Phi_{N; L}^{(\epsilon)}$ for this solution.  We will have
 $|R_{\epsilon}(x, t)|^2 \leq M_0^2$ on $M\times [0, \Omega_{\epsilon}]$ and 
\[
  \Phi^{(\epsilon)}_{N; L}(x, 0) = |R_{\epsilon}(x, 0)|^2 = |R(0, \epsilon)|^2 \leq M_0^2
\]
on $M$,
just as in the argument below, but, in addition, Shi's estimates (cf. \cite{Shi}) for the derivatives of the curvature tensor of $g(t)$ will imply uniform bounds of the form
\[
   \sup_{M\times [0, \Omega_{\epsilon}]} |\R{k}{l}_{\epsilon}| \leq C(\epsilon, k, l, n, M_0)
\]
 for the derivatives of the curvature tensor $R_{\epsilon}$ of $g_{\epsilon}$, and hence a bound of the form $\Phi^{(\epsilon)}_{N; L} \leq M_1(N, L)$
 for any fixed $L$ and $N$.  Although these bounds are $O(\epsilon^{-N})$ as $\epsilon\searrow 0$, 
using $\Phi^{(\epsilon)}_N$ in the argument below,
 we will be able to apply the maximum principle for any \emph{fixed} $\epsilon$ to obtain the $\epsilon$-independent bound 
 $\Phi^{(\epsilon)}_{N; L} \leq 2M_0^2$ on $M\times [0, \Omega_{\epsilon}]$. Sending $\epsilon \searrow 0$ then yields the same bound for $\Phi_{N; L}$ on 
 $M\times [0, \Omega]$. Note that the lower bound $L_2$ on $L$ from Proposition \ref{prop:phiev} we use here is independent of $\epsilon$, since we assume
$\Omega \leq 1$ and so have 
\[	
  \Omega_{\epsilon}^* = \max\{\Omega-\epsilon, 1\} = 1 = \Omega^*.
\]
Thus we assume below that each $\Phi_{N; L}$ is uniformly bounded on $M\times [0, \Omega_0]$.

 Let the constants $C_1$ and $L_2$ be as in Proposition \ref{prop:phiev} and $L \geq L_2$. 
We argue by induction on $N$.
  The case $N=0$ is trivial, as we have $\Phi_{0; L} = |R|^2 \leq M_0^2$ by assumption. Assume, then, that, for some $N\geq 1$,
  we have $\Phi_{N-1} = \Phi_{N-1; L} \leq 2M_0^2$ on $M\times [0, \Omega_0]$.  By Proposition \ref{prop:phiev}, we have
\begin{equation*}
   \left(\pdt - \Delta\right) \Phi_{N} \leq -(1-C_1\theta^2(\Phi_{N}^2+1))\Psi_N + C_1(M_0^2+1)(\Phi_{N}^2 + 1)
\end{equation*}
 on $M\times [0, \Omega]$. Now, $\Phi_N(x, 0) = |R(x, 0)|^2 \leq M_0^2$, and, as we have noted above, we may assume that
$\Phi_{N}$ is uniformly bounded, say, by $M_1$, on $M\times [0, \Omega_0]$.  If we define
\[
   \tau_N \dfn \tau_{N, L} \dfn \sup\bigg\{\,t\in [0, \Omega_0]\,\bigg|\, \sup_{x\in M} \frac{C_1t^2}{L^2}\left(\Phi_{N}^2(x, t)+ 1\right)\leq 1\,\bigg\} 
\]
then $\tau_{N}\geq L/\sqrt{C_1(M_1^2+1)} > 0$, and we have
\[
 \left(\pdt - \Delta\right)\Phi_N \leq A(\Phi_{N}^2 + 1)
\]
on $M\times [0, \tau_N]$, where $A \dfn C_1(M_0^2 + 1)$.  Therefore, if 
if $F(t)$ is the solution to $F^{\prime} = A(F^2 + 1)$ with $F(0) = M_0^2$,
  i.e., $F(t) = \tan(\delta + At)$
where $\delta \dfn \arctan(M_0^2) < \pi/2$, then, by the maximum principle, 
we have $\Phi_N(x, t) \leq F(t)$ on $M\times [0, \min\{\tau_N, \tau^*\}]$, where $\tau^* \dfn A^{-1}(\pi/2-\delta)$ is the upper bound
of the interval of existence for $F(t)$. (If $M$ is not compact, the use
of the maximum principle is justified by the uniform bounds on $\Phi_N$ and $R$;  see, e.g., Theorem 12.14 of \cite{RFV2P2} 
for a precise statement of the maximum principle we invoke.)  Since
\[
    \arctan{(2M_0^2)} - \arctan{(M_0^2)} \leq \frac{M_0^2}{1+M_0^2},
\]
we can  ensure $F(t) \leq 2M_0^2$ for $t \leq \Omega_1^{\prime} \dfn C_1^{-1}M_0^2/(1+M_0^2)^2$,
and hence that $\Phi_N \leq 2M_0^2$ on $M\times [0, \min\{\tau_N, \Omega^{\prime}\}]$.

We are still free to adjust $L$, and claim that, if $L > 4/\sqrt{C_1}$, then $\tau_{N} > \Omega^{\prime}$,
and, consequently, $\Phi_{N, L} \leq 2M_0^2$ on $M\times [0, \Omega^{\prime}]$. 
For, if $\tau_{N} \leq \Omega^{\prime}$, there would exist $(x_0, t_0) \in M\times (0, \tau_N)$ at which
\begin{align}\label{eq:taubd}
      \frac{C_1t_0^2}{L^2}(\Phi_{N}^2(x_0, t_0)+1)\geq \frac{1}{2}.
\end{align}
But, on $M\times (0, \tau_N) \subset M\times [0, \Omega^{\prime}]$, we have $\Phi_N(x, t)\leq 2M_0^2$, so
\begin{align*}
  \frac{C_1t_0^2}{L^2}(\Phi_{N}^2(x_0, t_0)+1) &\leq \frac{C_1^2(\Omega^{\prime})^2(4M_0^4+1)}{16}\leq \frac{M_0^4(M_0^4+1)}{4(M_0^2+1)^4} \leq \frac{1}{4},
\end{align*}
contradicting \eqref{eq:taubd}.
(Note that, for any $(x, t)$, $\Phi_{N; L}(x, t)$ is monotone-decreasing in $L$.)
Thus, for such $L$, $\tau_N = \tau_{N, L} > \Omega^{\prime}$, and we may take $\Omega_1 \dfn \Omega^{\prime}$ to obtain estimate \eqref{eq:phibound}.
\end{proof}

It remains to prove Proposition \ref{prop:phiev}. We carry this out over the next three sections.

\section{A preliminary estimate on $\left(\pdt - \Delta\right)\Phi_N$.}\label{sec:phiev}
As a first step toward the proof of Proposition \ref{prop:phiev}, we perform some rather high-level manipulations on the evolution
equation for $\Phi_N$, leaving a detailed consideration of the commutation and reaction terms for the next section. 
For now, we introduce notation for these terms and bundle them together. For any $k$, $l\in \NN_0$, we define
\[
     Q_{k, l}\dfn \left(\pdt - \Delta\right)\R{k}{l},
\]
and further split $Q_{k, l}$ into three terms 
\[
    Q_{k, l} = U_{k, l} + V_{k, l} + W_{k, l}
\]
where
\begin{align}\begin{split}\label{eq:uvwdef}
  U_{k, l} &\dfn \nabla^{(k)}\dt{l}\left(\pdt -\Delta\right) R =  \nabla^{(k)}\dt{l}Q_{0, 0},\\
  V_{k, l} &\dfn \left[ \left(\pdt - \Delta\right), \nabla^{(k)}\right]\R{0}{l},\quad\mbox{and} \\
   W_{k, l} &\dfn \nabla^{(k)}\left[\left(\pdt - \Delta\right), \dt{l}\right] R = \nabla^{(k)}\left[\dt{l}, \Delta\right]R.
\end{split}
\end{align}
It will be helpful also to set aside notation for the scaled norms of the above quantities; we will write
\begin{align}\label{eq:mcuvwdef}
\begin{split}
\mc{Q}_{k, l} &\dfn a_{k, l}^2\theta^{k+2l}|Q_{k, l}|^2, \quad \mc{U}_{k, l} \dfn a_{k, l}^2\theta^{k+2l}|U_{k, l}|^2,\\
  \mc{V}_{k, l} &\dfn a_{k, l}^2\theta^{k+2l}|V_{k, l}|^2, 
    \quad\mc{W}_{k, l}\dfn  a_{k, l}^2\theta^{k+2l}|W_{k, l}|^2,
\end{split}
\end{align}
and
\[
      \mc{S}_{k, l} = \mc{U}_{k, l} + \mc{V}_{k, l} + \mc{W}_{k, l},
\]
and observe now the elementary inequality
\begin{equation}\label{eq:qsineq}
    \mc{Q}_{k, l} \leq 3\mc{S}_{k, l}
\end{equation}
for later use.

Our goal in this section is to establish the following result. 
\begin{proposition}\label{prop:prelimphibound}
  Under the assumptions of Theorem \ref{thm:mainest}, there exist constants $C_3 = C_3(n)$ and $L_4 = L_4(n, M_0, \Omega^*)$
such that, whenever $L \geq L_4$, for any $N\in \NN_0$, the evolution of the quantity $\Phi_N = \Phi_{N; L}$ satisfies
\begin{equation}\label{eq:prelimphibound}
 \left(\pdt-\Delta\right)\Phi_N \leq -\frac{3}{2}\Psi_{N} + \sum_{1\leq k+l\leq N}\left(\frac{C_3\theta}{[k-1]^2}\mc{S}_{k, l}\right)
      + C_3M_0\phi_{0, 0}
\end{equation}
on $M\times [0, \Omega]$.
\end{proposition}

\subsection{The evolution of $\phi_{k, l}$.}

We begin with two simple computations.
\begin{lemma}
 For any $k$, $l\in \NN_0$,  we have
\begin{align}\begin{split}\label{eq:rklevol}
 \left(\pdt - \Delta\right)|\R{k}{l}|^2 &\leq -2|\R{k+1}{l}|^2  + 4\sqrt{n}(k+4)|R||\R{k}{l}|^2 + 2\left\langle \R{k}{l}, Q_{k, l}\right\rangle
\end{split}
\end{align}
on $M\times [0, \Omega]$.
\end{lemma}
\begin{proof}
  For any $k$ and $l$, the tensor $\R{k}{l} = \nabla^{(k)}\dt{l} R$ is a family of sections of $T_{k+3}^{1}(M)$, and
  so
  \[
	|\R{k}{l}|^2 \lt g\ast \underbrace{h\ast h \ast \cdots \ast h}_{k+3}\ast \R{k}{l}\ast \R{k}{l},
  \]
  which implies
  \begin{align*}
  \begin{split}
        \left(\pdt-\Delta\right) |\R{k}{l}|^2 &\leq -2\left|\nabla \R{k}{l}\right|^2+  4(k+4)|\Rc||\R{k}{l}|^2 \\
	    &\phantom{\leq}+ 2\left\langle \left(\pdt-\Delta\right) \R{k}{l}, \R{k}{l} \right\rangle, 
  \end{split}
  \end{align*}
  and \eqref{eq:rklevol} follows.
\end{proof}

\begin{lemma}\label{lem:phiklevol}
For any $L > 0$ and any $k$, $l\in \NN_0$, the evolution of 
$\phi_{k, l} = \phi_{k, l; L}$ can be estimated as
\begin{align}
\begin{split}\label{eq:phiklevol}
 \left(\pdt - \Delta\right)\phi_{k, l} &\leq -2\psi_{k+1, l} + \left(4\sqrt{n}(k+4)|R| + \frac{(k+2l)}{L\theta}\right)\phi_{k, l}
  \\
   &\phantom{\leq} + 2\sqrt{\mc{Q}_{k, l}\phi_{k, l}}
\end{split}
\end{align}
on $M\times (0, \Omega]$.
\end{lemma}
\begin{proof}
For any $0 < t \leq \Omega$, we have
\begin{align*}
    \left(\pdt - \Delta\right)\phi_{k, l} &= \frac{k+2l}{t}\phi_{k, l} + a^2_{k, l}\theta^{k+2l}\left(\pdt -\Delta\right)|\R{k}{l}|^2\\
	&\leq  -2a_{k, l}^2\theta^{k+2l}|\R{k+1}{l}|^2 + \left(4\sqrt{n}(k+4)|R| + \frac{k +2l}{L\theta}\right)\phi_{k, l}\\
	 &\phantom{\leq}+ 2a_{k, l}^2\theta^{k+2l}\langle \R{k}{l}, Q_{k, l}\rangle,
\end{align*}
from which \eqref{eq:phiklevol} follows, since 
\[
a_{k, l}^2\theta^{k+2l}|\R{k+1}{l}|^2 = \frac{[k-1]^2}{\theta}\phi_{k+1, l} =\psi_{k+1, l},
\]
and
\begin{align*}
    2a_{k, l}^2\theta^{k+2l}|\R{k}{l}||Q_{k, l}| = 2\sqrt{\mc{Q}_{k, l}\phi_{k, l}}.
\end{align*}
\end{proof}

Although the overall strategy in this section is much the same as in \cite{Bando}, in that, in the evolution equation for 
$\Phi_N$, we seek to exploit
the ``good'' negative  $\psi_{k, l} = [k-1]^2\phi_{k+1, l}/\theta$ terms, arising as in the right-hand side of \eqref{eq:phiklevol}, to control
all of the other ``bad'' positive terms whose coefficients contain additional factors of $k$ and $l$, and are not simply proportional
by a constant to some term $\phi_{k, l}$.  From our perspective, of the bad terms
that are presently visible in \eqref{eq:phiklevol} (and not concealed in $\mc{Q}_{k, l}$)
the most problematic is perhaps $(k+2l)\phi_{k, l}/t$,
which arises when the operator $\pdt$ falls on the factor of $\theta^{k + 2l}$.  When 
$k$ is small relative to $\sqrt l$, this term cannot
be controlled by a corresponding good term in $(\pdt -\Delta)\Phi_N$,  as the  coefficients of this good term will
be only proportional to  $k^2$.  Moreover, for $k=0$ (i.e., for the pure time-derivatives), 
there are \emph{no} corresponding good (negative) terms to balance these contributions. In
Bernstein-type arguments, such good terms arise from the application of the operator $-\Delta$ to the square of (the norm of) the solution,
and always involve at least one spatial derivative.  Fortunately, there is a 
workaround, but we will need to use two different estimates
for the evolutions of the $\phi_{k, l}$, depending on whether $k$ is large or small relative to $\sqrt{l}$. 

In the case that $k$ is small relative to $\sqrt{l}$, the idea will be to exchange one time-derivative for
 two space derivatives while playing one coefficient off the other; the basis of the estimate in this case is the following
simple inequality.
\begin{lemma}\label{lem:phi2dbd} For any $k\in \NN_{0}$, $l\in \NN$, and $L > 0$, $\phi_{k, l} = \phi_{k, l; L}$
satisfies
\begin{equation}\label{eq:phi2dbd}
 \phi_{k, l} \leq \frac{2n\theta[k-1]^2}{[l-2]^2}\psi_{k+2, l-1} + \frac{4\theta^2}{[l-2]^2}\mc{S}_{k, l-1}
\end{equation}
on $M\times [0, \Omega]$.
\end{lemma}
\begin{proof}
    To begin with, we have 
\begin{equation}\label{eq:phi2dbd1}
      \R{k}{l} = \nabla^{(k)}\pdt\left(\R{0}{l-1}\right) = \nabla^{(k)}\left(\Delta\R{0}{l-1} + Q_{0, l-1}\right)
\end{equation}
for any for $l > 0$. Now, recalling \eqref{eq:uvwdef}, we can write
\begin{align*}
  \nabla^{(k)}Q_{0, l-1} &=\nabla^{(k)}\left(\left(\pdt - \Delta\right)\R{0}{l-1}\right)\\
      &= \left[\nabla^{(k)}, \left(\pdt-\Delta\right)\right]\R{0}{l-1} + \left(\pdt - \Delta\right)\nabla^{(k)}\R{0}{l-1}\\
      &= - V_{k, l-1} + Q_{k, l-1}\\
      &= U_{k, l-1} + W_{k, l-1},
\end{align*}
so, together with \eqref{eq:phi2dbd1}, we have 
\begin{align*}
    |\R{k}{l}|^2 &\leq 2|\nabla^{(k)}\Delta\R{0}{l-1}|^2 + 2|\nabla^{(k)}Q_{0, l-1}|^2\\
		 &\leq 2n|\R{k+2}{l-1}|^2 + 4|U_{k, l-1}|^2 + 4|W_{k, l-1}|^2,
\end{align*}
and, adding the weights $a_{k, l}$,
\begin{align*}
  \phi_{k, l} &\leq \frac{2n[k]_2^2}{[l-2]^2}\phi_{k+2, l-1} + \frac{4\theta^2}{[l-2]^2}\left(\mc{U}_{k, l-1} + \mc{W}_{k, l-1}\right)\\
	      &\leq \frac{2n\theta[k-1]^2}{[l-2]^2}\psi_{k+2, l-1} + \frac{4\theta^2}{[l-2]^2}\mc{S}_{k, l-1}.
\end{align*}
\end{proof}

We now combine the above two lemmas and give estimates for $\phi_{k, l}$ specialized to the relative sizes of
$k$ and $l$.
\begin{lemma}\label{lem:phikl2ndest}
Given any $M_0$ and $\Omega > 0$ and a solution $g(t)$ to \eqref{eq:rf} on $M\times [0, \Omega]$
satisfying 
\[
  \sup_{M\times [0, \Omega]}|R(x, t)|\leq M_0,
\]
there exist constants $C_3 = C_3(n)$ and  $L_4 = L_4(n, M_0, \Omega^*)$ such that for any $L \geq L_4$,
the evolution of $\phi_{k, l} = \phi_{k, l; L}$ can be estimated as follows: For $k =0$ and $l = 0$, 
\begin{equation}\label{eq:phi00est}
  \left(\pdt- \Delta\right)\phi_{0, 0} \leq -2\psi_{1, 0} + C_3M_0\phi_{0, 0}.
\end{equation}
 For $l > 0$ and $0 \leq k < \sqrt l$,
\begin{align}
\begin{split}\label{eq:phiestksmall}                 
           \left(\pdt - \Delta\right)\phi_{k, l}&\leq -2\psi_{k+1, l} + \frac{1}{4}\psi_{k+2, l-1}
  + \frac{C_3\theta}{[k-1]^2}\left(\mc{S}_{k, l} + \mc{S}_{k, l-1}\right).
\end{split}
\end{align}
For $k > 0$ and $0 \leq \sqrt{l} \leq k$,
\begin{equation}\label{eq:phiestklarge}
  \left(\pdt - \Delta\right)\phi_{k, l} \leq -2\psi_{k+1, l} + \frac{1}{4}\psi_{k, l} + \frac{C_3\theta}{[k-1]^2}\mc{S}_{k, l}.
\end{equation}
\end{lemma}
\begin{proof}
    We take $C_3 \dfn 10000n^2$, $L_4 \dfn 8 C_3(M_0\Omega^* + 1)$, assume $L \geq L_4$, and consider each case in turn.  

\noindent{\it The case $k = l=0$:} Using the standard equation for the squared norm of the Riemann curvature tensor under the Ricci flow (see, e.g., Section 7.2 of \cite{ChowKnopf}),
we have
\[
    \left(\pdt - \Delta\right)\phi_{0, 0} \leq -2|\nabla R|^2 + 16M_0\phi_{0, 0} = -2\psi_{1, 0} + 16 M_0\phi_{0, 0},
\]
which is \eqref{eq:phi00est}.

\noindent{\it The case $l > 0$ and $0 \leq k < \sqrt{l}$:} Observe that we only need to consider the last two terms
 of \eqref{eq:phiklevol} to obtain \eqref{eq:phiestksmall}.  We start with the second term.  By Lemma \ref{lem:phi2dbd}, we have
\begin{equation}\label{eq:phismallk}
  \phi_{k, l} \leq 4n\frac{\theta[k-1]^2}{[l-2]^2}\left(\psi_{k+2, l-1} + \frac{\theta}{[k-1]^2}\mc{S}_{k, l-1}\right),
\end{equation}
and since, when $k < \sqrt{l}$,
\[
    \frac{(k+4)[k-1]^2}{[l-2]^2} \leq \left(\frac{l+4}{[l-2]}\right)^2 \leq 49, \quad\mbox{and}\quad 
  \frac{[k-1]^2(k+2l)}{[l-2]^2}\leq 3\left(\frac{l+1}{[l-2]}\right)^2 \leq 48,
\]
 we therefore have
\begin{align}
\nonumber   &\left(4\sqrt{n}(k+4)|R| + \frac{(k+2l)}{L\theta}\right)\phi_{k, l}\\
\nonumber  &\qquad\leq\frac{16n^2\theta[k-1]^2}{[l-2]^2}\left((k+4)|R| + \frac{(k+2l)}{L\theta}\right)
      \left(\psi_{k+2, l-1} + \frac{\theta}{[k-1]^2}\mc{S}_{k, l-1}\right)\\
\nonumber     &\qquad \leq \frac{800n^2}{L}\left(M_0\Omega^* + 1\right)\left(\psi_{k+2, l-1} + \frac{\theta}{[k-1]^2}\mc{S}_{k, l-1}\right)\\
\label{eq:ksmallterm1}     &\qquad\leq \frac{1}{8}\psi_{k+2, l-1} + \frac{\theta}{8[k-1]^2}\mc{S}_{k, l-1},
\end{align}
using our assumption on $L_4$. 

For the last term in \eqref{eq:phiklevol}, using \eqref{eq:qsineq}, we can estimate
\begin{align}\label{eq:ipterm}
  2\sqrt{\phi_{k, l}\mc{Q}_{k, l}}\leq 
  4\sqrt{\phi_{k, l}\mc{S}_{k, l}} 
      &\leq \epsilon\frac{[k-1]^2}{\theta}\phi_{k, l} + \frac{4\theta}{\epsilon[k-1]^2}\mc{S}_{k, l},      
\end{align}
for any $\epsilon > 0$, 
and use \eqref{eq:phismallk} as before to see that 
\begin{align*}
    \frac{[k-1]^2}{\theta}\phi_{k, l} &\leq 4n\left(\frac{[k-1]^2}{[l-2]}\right)^2\left(\psi_{k+2, l-1} + \frac{\theta}{[k-1]^2}\mc{S}_{k, l-1}\right)\\
				     &\leq 64n\left(\psi_{k+2, l-1} + \frac{\theta}{[k-1]^2}\mc{S}_{k, l-1}\right),
\end{align*}
as
\[
    \left(\frac{[k-1]^2}{[l-2]}\right)^2 \leq \left(\frac{l+1}{[l-2]}\right)^2 \leq 16
\]
when $k < \sqrt{l}$.  Thus, if $\epsilon < 1/512n$, we have
\[
  \frac{\epsilon[k-1]^2}{\theta}\phi_{k, l} \leq \frac{1}{8}\psi_{k+2, l-1} + \frac{\theta}{8[k-1]^2}\mc{S}_{k, l-1},
\]
and hence, from \eqref{eq:ipterm},
\begin{align}
\label{eq:ksmallterm2}
    2\sqrt{\phi_{k, l}\mc{Q}_{k, l}} &\leq
\frac{1}{8}\psi_{k+2, l-1} + \frac{\theta}{8[k-1]^2}\mc{S}_{k, l-1} +
		\frac{2048n\theta}{[k-1]^2}\mc{S}_{k, l}.
\end{align}
Taking \eqref{eq:phiklevol} with \eqref{eq:ksmallterm1} and \eqref{eq:ksmallterm2}, we see that, for $l > 0$ and $0 \leq k < \sqrt{l}$, we have
\[
  \left(\pdt - \Delta\right)\phi_{k, l} \leq -2\phi_{k+1, l} + \frac{1}{4}\psi_{k+2, l-1}  +\frac{C_3\theta}{[k-1]^2}\left(\mc{S}_{k, l} + \mc{S}_{k, l-1}\right),
\]
which is \eqref{eq:phiestksmall}.

\noindent{\it The case $k > 0$ and $0\leq \sqrt{l} \leq k$:} The estimate is easier for $k$ in this range. As before, we just
need to estimate the last two terms in \eqref{eq:phiklevol}.  The second term on the right-hand side of \eqref{eq:phiklevol}
can be estimated by
\begin{align}
\nonumber  &\left(4n(k+4)|R| + \frac{(k+2l)}{L\theta}\right)\phi_{k, l} \\
\nonumber  &\qquad\qquad\leq 4n\left(\frac{(k+4)}{[k-2]^2}|R|\theta + \frac{k^2}{L[k-2]^2}\right)\frac{[k-2]^2}{\theta}\phi_{k, l}\\
\nonumber  &\qquad\qquad\leq \frac{36n}{L}\left(M_0\Omega^* + 1\right)\psi_{k, l}\\
\label{eq:klargeterm1}  &\qquad\qquad\leq \frac{1}{8}\psi_{k, l}
\end{align}
for $t >0$, provided $L \geq L_4$, and the last term from the same equation can be estimated simply by
\begin{equation}\label{eq:klargeterm2}
  2\sqrt{\phi_{k, l}\mc{Q}_{k, l}} \leq  \frac{[k-2]^2}{8\theta}\phi_{k, l} + \frac{8\theta}{[k-2]^2}\mc{Q}_{k, l}
\leq  \frac{1}{8}\psi_{k, l} + \frac{24\theta}{[k-1]^2}\mc{S}_{k, l},
\end{equation}
using \eqref{eq:psiphirel} and \eqref{eq:qsineq} again.
Taken together, \eqref{eq:phiklevol}, \eqref{eq:klargeterm1} and \eqref{eq:klargeterm2}  yield
\[
  \left(\pdt - \Delta\right)\phi_{k, l} \leq -2\phi_{k+1, l} + \frac{1}{4}\psi_{k, l} + \frac{C_3\theta}{[k-1]^2}\mc{S}_{k, l},
\]
which is \eqref{eq:phiestklarge}.
\end{proof}

\subsection{Proof of Proposition \ref{prop:prelimphibound}}

Now we put everything together we have so far to obtain a preliminary estimate on the evolution of $\Phi_{N}$.

\begin{proof}  Let $C_3$ and $L_4$ be as in Lemma \ref{lem:phikl2ndest}. We choose $C_2 \geq 2C_3$, $L_3\geq L_4$, 
and assume that $L \geq L_3$.The case $N = 0$ follows immediately from $\eqref{eq:phi00est}$, so we assume that $N \geq 1$. 
We write $\Phi_N = \Phi_{N, L}$ as
\[
\Phi_{N} = \phi_{0, 0} + \sum_{\stackrel{0 < k + l\leq N}{0\leq k < \sqrt{l}}}\phi_{k, l} 
      +\sum_{\stackrel{0< k + l\leq N}{0\leq \sqrt{l} \leq k}} \phi_{k, l}
\]
and then combine the estimates in Lemma \ref{lem:phikl2ndest} to obtain
\begin{align*}
\begin{split}
 &\left(\pdt - \Delta\right)\Phi_N\\
&\qquad\leq -2\sum_{0\leq k+l\leq N} \psi_{k+1, l} + \sum_{\stackrel{0 < k + l\leq N}{0\leq k < \sqrt{l}}}\left(\frac{1}{4}\psi_{k+2, l-1}
	+ \frac{C_3\theta}{[k-1]^2}\left(\mc{S}_{k, l} + \mc{S}_{k, l-1}\right)\right)\\
 &\qquad\phantom{\leq} + \sum_{\stackrel{0< k + l\leq N}{0\leq \sqrt{l} \leq k}}\left( \frac{1}{4}\psi_{k, l} 
	+ \frac{C_3\theta}{[k-1]^2}\mc{S}_{k, l}\right)+C_3M_0\phi_{0, 0}.
\end{split}
\end{align*}
Now, by re-indexing (and using, for example, that $k\leq \sqrt{l}$ and $1\leq l \leq N$ implies $k+2\leq N +1$), we see that
\[
  \left(\pdt - \Delta\right)\Phi_N \leq -\frac{3}{2}\Psi_N + 2C_3\theta\sum_{m=1}^{N}\sum_{k=0}^{m}\frac{\mc{S}_{k, m-k}}{[k-1]^2} +  C_3M_0\phi_{0, 0}, 
\]
which is \eqref{eq:prelimphibound}.
\end{proof}

\section{Estimation of the reaction and commutator terms.}

To prove Proposition \ref{prop:phiev}, it remains to estimate the term
\[
  \sum_{m=1}^{N}\sum_{k=0}^{m}\frac{\theta\mc{S}_{k, m-k}}{[k-1]^2}
\]
from \eqref{eq:prelimphibound}. We split this sum into the three terms
corresponding to $\mc{U}_{k, l}$, $\mc{V}_{k, l}$, and $\mc{W}_{k, l}$, and consider each in turn.
First, we use the induction hypothesis to obtain bounds on the derivatives of the inverse metric
whose factors crop up in the formulas for these terms.
In view of Lemma \ref{lem:prodsum}, we only need to bound the weighted sum of these derivatives. 
\subsection{Inductive bounds on the derivatives of the inverse of the metric.}

The next lemma is nothing more than a quantitative statement of the following two basic observations:
first, any iterated covariant derivative of the $l$-th time derivative of $g$ 
can be controlled by the same covariant derivative of the $l-1$-th time derivative of $R$ in view of \eqref{eq:rf},
and, second, bounds on the mixed derivatives of $g$ up to a total order $N$ derived from \eqref{eq:mainestvar}
 imply bounds of an identical form on the same derivatives of $h$ 
via the identity $h^{ik}g_{jk} = \delta_{j}^i$.  That the implied estimates in this latter observation are of same form (i.e., same order of magnitude)
as those on $g$, is made possible by the ``factorial lag'' built into the coefficients $a_{k, l}$.

\begin{lemma}\label{lem:invmetbounds} For any $M$, $\Omega$, and $\delta > 0$, there exists
a constant
$L_5 = L_5(n, M, \Omega^*, \delta)$ such that whenever $L \geq L_5$ and 
\[
  \sup_{M\times [0, \Omega]} \Phi_{N; L}(x, t) \leq M,
\]
for some $N\in \NN_{0}$, then
\[
    \sup_{M\times [0, \Omega]} X_{N+1; L}(x, t) \leq (1+\delta)n.
\]
\end{lemma}
\begin{proof} For now, we will take $L_5$ to be a large positive constant whose value (to depend only on $n$, $\Omega^*$, and $M$)
 we will specify later. We assume that $L\geq L_5$ and argue by induction.  Note that $X_{0; L}(x, t) = |h|^2(x, t) \equiv n$.
Suppose then that we have 
\[
    \sup_{M\times [0, \Omega]}X_{P; L} \leq (1+\delta)n
\]
for some $0\leq P \leq N$.
Applying the operator $\nabla^{(k)}\dt{l}$ to the identity $h^{ik}g_{kj} = \delta^i_j$, and using that
$\h{p}{0} = \nabla^{(p)}h \equiv 0$ and $g^{(p, 0)} = \nabla^{(p)}g \equiv 0$ when $p > 0$, we obtain that
\begin{align*}
  |\h{k}{l}| &\leq |g^{(k, l)}| + \frac{1}{\sqrt{n}}\sum_{p=0}^k\sum_{q=1}^{l-1}\binom{k}{p}\binom{l}{q}|\h{p}{q}||g^{(k-p, l-q)}|\\
	     &\leq 2\sqrt{n}|\R{k}{l-1}| + 2\sum_{p=0}^k\sum_{q=1}^{l-1}\binom{k}{p}\binom{l}{q}|\h{p}{q}||R^{(k-p, l-q-1)}|
\end{align*}
for $k\in \NN_0$ and $l \in \NN$. So we have $\chi_{0, 0} = |h| = \sqrt n$, and
\begin{align}
\nonumber \chi_{k, l}^{1/2} &\leq a_{k, l}\theta^{k/2+l}|\h{k}{l}|\\
\nonumber	  &\leq \frac{2\sqrt{n}\theta}{[l-2]}\phi^{1/2}_{k, l-1} +
		      2\theta \sum_{p=0}^k\sum_{q=1}^{l-1}\frac{[k]_2[l]_2}{[p]_2[k-p]_2[q]_2[l-q]_3}\chi_{p, q}^{1/2}\phi^{1/2}_{k-p,l-q-1}\\
\nonumber	  &\leq C\theta\sum_{p=0}^k\sum_{q=0}^{l-1}\frac{[k]_2[l]_2}{[p]_2[k-p]_2[q]_2[l-q]_2}\chi_{p, q}^{1/2}\phi^{1/2}_{k-p,l-q-1},
\end{align}
for some $C = C(n)$. Since, as usual,
\[
      \sum_{p=0}^k\sum_{q=0}^{l-1}\left(\frac{[k]_2[l]_2}{[p]_2[k-p]_2[q]_2[l-q]_2}\right)^2 \leq
C \sum_{p=0}^k\sum_{q=0}^{l}\frac{[k]_4[l]_4}{[p]_4[k-p]_4[q]_4[l-q]_4}\leq C 
\]
by Lemma \ref{lem:comb1}, we have
\begin{align}\label{eq:chiest}	  
    \chi_{k, l} &\leq C\theta^2\sum_{\stackrel{p+q = k}{r+s = l-1}}\chi_{p, r}\phi_{q,s},
\end{align}
for another constant $C = C(n)$.

Fixing any $(x, t)\in M\times [0, \Omega]$,  
summing over all $k$ and $l$ such that $0\leq k + l \leq P+1$, and using again that $\chi_{0, 0} \equiv n$ and $\chi_{k, 0} \equiv 0$ for $k > 0$,
we have
\begin{align*}
 X_{P+1} &= n +\sum_{0\leq k + l \leq P}\chi_{k, l+1},
\end{align*}
and so, by \eqref{eq:chiest}, that
\begin{align*}
      X_{P+1}  &\leq n + C\theta^2\sum_{0\leq k+l \leq P}\sum_{\stackrel{p+q = k}{r+s = l}}\chi_{p, r}\phi_{q, s}\\
       &\leq n + C\theta^2X_{P}\Phi_{P}.
\end{align*}
Thus, by the induction hypothesis,
\[
       X_{P+1} \leq n\left(1 + \frac{(1+\delta)Ct^2M}{L^2}\right),
\]
from which we obtain the desired inequality, for example, with the choice 
\[
  L_5 \dfn \Omega^*\sqrt{\frac{(1+\delta)CM}{\delta}}.
\]
\end{proof}

\subsection{The reaction term $\mc{U}_{k, l}$}

We begin by recalling the evolution equation for the $(3, 1)$ curvature tensor $R$ under the Ricci flow (cf., e.g.,
\cite{ChowKnopf}, p. 177):
\begin{align}
\begin{split}\label{eq:revol}
  \left(\pdt-\Delta\right) R_{ijk}^l &= h^{rs}\left(R_{ijr}^pR_{psk}^l - 2R_{rik}^pR_{jsp}^l + 2R_{rip}^lR_{jsk}^p\right)\\
		&\phantom{=} -h^{rs}\left(R^p_{irs}R_{pjk}^l + R^p_{jrs}R_{ipk}^l + R_{krs}^pR_{ijp}^l - R_{prs}^lR_{ijk}^p\right).
              \end{split}
\end{align}
According to our convention, we may express this as
\begin{equation}\label{eq:revol2}
  \left(\pdt-\Delta\right)R = Q_{0, 0} \lt  (4n+5) h\ast R\ast R.
\end{equation}

\begin{proposition}
 There exists $C_4 = C_4(n)$ such that, for any $N\in \NN_0$ and any $L > 0$,
$X_N  = X_{N; L}$, $\Phi_N = \Phi_{N; L}$, and the quantity $\mc{U}_{k, l}$ associated
to the evolution of $\Phi_N$ satisfy
\begin{equation}\label{eq:uklest}
    \sum_{0\leq k+l\leq N}\mc{U}_{k, l} \leq C_4 X_N \Phi_N^2
\end{equation}
on $M\times [0, \Omega]$.
\end{proposition}
\begin{proof}
   Put $A = h\ast R\ast R$. For any $k$, $l\in \NN_0$, we have 
\[
    |U_{k, l}|^2 \leq \left|\nabla^{(k)}\dt{l}Q_{0, 0}\right|^2 \leq (4n + 5)^2 |A^{(k, l)}|^2,
\]
by \eqref{eq:revol2}. Iterating the result of Lemma \ref{lem:prodsum} implies that there exists a constant $C = C(n)$
such that
\begin{align*}
  \sum_{0\leq k+l\leq N} a_{k, l}^2\theta^{k+2l}|\A{k}{l}|^2 
      &\leq C\left(\sum_{0\leq k+l \leq N} \chi_{k, l}\right)\left(\sum_{0\leq k+l \leq N}\phi_{k, l}\right)^2\\
      &\leq CX_N \Phi_N^2.
\end{align*}
Thus \eqref{eq:uklest} follows with $C_4 \geq (4n+5)^2C$.
\end{proof}

\subsection{The commutator term $\mc{V}_{k, l}$}

Here we borrow a formula from \cite{Bando} and \cite{RFV2P2}
 for the commutator of the $k$-fold covariant derivative with the heat operator associated
to a solution of the Ricci flow.

\begin{lemma}[\cite{Bando}; \cite{RFV2P2}, Corollary 13.27]\label{lem:nablakheatcomm}
If $T$ is a tensor of rank $\rho$, and $g(t)$ is a solution to the Ricci flow, then
\begin{equation*}
  \left|\left[\nabla^{(k)}, \left(\pdt-\Delta\right)\right] T\right| \leq 3n(\rho+1)\sum_{i=0}^k(i+k+4)\frac{(k+1)!}{(i+2)!(k-i)!}|\R{i}{0}||\T{k-i}{0}|.
\end{equation*}
\end{lemma}
Thus, for any $k$, $l\in \NN_0$, we can apply the above to $T = \R{0}{l}$ to obtain
\begin{align}
\begin{split}\label{eq:vkl}
     |V_{k, l}| \leq 30n\sum_{i=0}^k\binom{k+2}{i+2}|\R{i}{0}||\R{k-i}{l}|.
\end{split}
\end{align}

\begin{proposition}\label{prop:vklest}
  There exists a constant $C_5= C_5(n)$ such that, for any $N\in \NN_0$ and $L > 0$, 
\begin{equation}\label{eq:vklest}
  \sum_{1\leq k + l \leq N}\frac{\mc{V}_{k, l}}{[k-1]^2} \leq C_5\theta\Phi_{N}\Psi_N.
\end{equation}
\end{proposition}
\begin{proof} Observe that $V_{0, l} = 0$, so we may assume $k\geq 1$ in what follows.
Starting with \eqref{eq:vkl}, and using  $C$ to denote a series of constants depending only on $n$, 
we have
\begin{align*}
  \mc{V}_{k, l} 
	        &\leq Ca_{k, l}^2\theta^{k+2l}\left(\sum_{i=0}^{k}\binom{k+2}{i+2}|\R{i}{0}|
|\R{k-i}{l}| \right)^2\\
    &\leq \frac{C}{[l]_2^2}
	\left(\sum_{i=0}^{k}\frac{[k+2]_4}{[i+2]_4[k-i]_2}\sqrt{\phi_{i, 0}\phi_{k-i, l}}\right)^2\\
    &\leq \frac{C}{[l]_2^2}\phi_{k, 0}\phi_{0, l} + \frac{C}{[l]_2^2}
	\left(\sum_{i=0}^{k-1}\frac{[k+2]_4}{[i+2]_4[k-i]_2}\sqrt{\phi_{i, 0}\phi_{k-i, l}}\right)^2\\
    &\leq \frac{C\theta}{[k-2]^2[l]^2_2}\psi_{k, 0}\phi_{0, l} + 
	\frac{C\theta}{[l]_2^2}\left(\sum_{i=0}^{k-1}\frac{[k+2]_4}{[i+2]_4[k-i]_2[k-i-2]}\sqrt{\phi_{i, 0}\psi_{k-i, l}}\right)^2.
\end{align*}
Using Cauchy-Schwarz and Lemma \ref{lem:comb1}, and arguing as in Lemma \ref{lem:prodsum},
the square in the second term can be estimated as
\begin{align*}
 &\left(\sum_{i=0}^{k-1}\frac{[k+2]_4}{[i+2]_4[k-i]_2[k-i-2]}\sqrt{\phi_{i, 0}\psi_{k-i, l}}\right)^2\\
      &\qquad\qquad\leq \left(\sum_{i=0}^{k-1}\frac{[k+2]_4^2}{[i+2]_4^2[k-i]_2^2[k-i-2]^2}\right)
	\left(\sum_{i=1}^{k}\phi_{k-i, 0}\psi_{i, l}\right)\\
      &\qquad\qquad\leq C[k]_8\left(\sum_{i=0}^{k-1}\frac{1}{[i]_8[k-i]_6}\right)\left(\sum_{i=1}^{k}\phi_{k-i, 0}\psi_{i, l}\right)\\
      &\qquad\qquad\leq C[k-1]^2\left(\sum_{i=1}^{k}\phi_{k-i, 0}\psi_{i, l}\right).
  \end{align*}
Thus, using $\mc{V}_{0, l} \equiv  0$, we have
\begin{align*}
   &\sum_{1\leq k + l\leq N}\frac{\mc{V}_{k, l}}{[k-1]^2} =\sum_{m=1}^N\sum_{k=1}^m\frac{\mc{V}_{k, m-k}}{[k-1]^2}\\ 
&\quad \leq C\theta
    \sum_{m=1}^N\sum_{k=1}^m\left(\psi_{k, 0}\phi_{0, m-k} 
      +  \sum_{i=1}^{k}\frac{\phi_{k-i, 0}\psi_{i, m-k}}{[m-k]_2^2}\right)\\
 &\quad\leq C\theta\left(\sum_{m=1}^N\psi_{m, 0}\right)\left(\sum_{m=0}^{N-1}\phi_{0, m}\right) 
       + C\theta\left(\sum_{m=1}^N\sum_{k=0}^{m-1}\frac{\phi_{k, 0}}{[m-k]_2^2}\right)
     \left(\sum_{m=1}^N\sum_{k=1}^m \psi_{k, m-k}\right)\\
 &\quad \leq C\theta\Psi_{N-1}\Phi_{N-1} 
  + C\theta\left(\sum_{m=0}^{N-1}\phi_{m, 0}\right)\left(\sum_{m=1}^{N}\frac{1}{m^4}\right)\Psi_{N-1}\\
 &\quad\leq C\theta\Psi_{N-1}\Phi_{N-1},
\end{align*}
 for some appropriately large $C = C(n)$, which we may take for our $C_5$. Since $\Phi_{N-1} \leq \Phi_N$, and 
$\Psi_{N-1} \leq \Psi_N$,
this completes the proof.
\end{proof}

\subsection{The commutator term $\mc{W}_{k, l}$}
Now we consider the most complicated component of $\mc{Q}_{k, l}$. We will first need a formula
for the commutator of an $l$-fold time-derivative with the tensor Laplacian associated to $g(t)$.
\begin{lemma}\label{lem:dtldeltacomm}
If $T$ is a smooth family of tensors of rank $\rho$, then, there exists
a constant $C= C(\rho, n)$ such that, for any $l\in \NN_0$, we have
\begin{align*}\begin{split}
    \left[\dt{l}, \Delta\right]T &\lt 
      C\bigg\{\sum_{\stackrel{|\gamma| = l}{\gamma_2 \geq 1}}
	\binom{l}{\gamma}\left(h^2\right)^{(0, \gamma_1)}\ast \R{1}{\gamma_2-1} \ast \T{1}{\gamma_3}\\
  &\phantom{\lt}+\sum_{\stackrel{|\gamma|=l}{\gamma_2, \gamma_3 \geq 1}}
	  \binom{l}{\gamma} \left(h^3\right)^{(0, \gamma_1)}\ast \R{1}{\gamma_2-1} \ast\R{1}{\gamma_3-1}
		\ast \T{0}{\gamma_4} \\
  &\phantom{\lt}+ \sum_{\stackrel{|\gamma| = l}{\gamma_2\geq 1}}\binom{l}{\gamma}\left(h^2\right)^{(0, \gamma_1)}
	  \ast \R{2}{\gamma_2-1}\ast \T{0}{\gamma_3}\\
  &\phantom{\lt}+ \sum_{|\gamma|=l-1}\binom{l}{\gamma}\h{0}{\gamma_1+1}\ast \T{2}{\gamma_2}\bigg\},
\end{split}
\end{align*}
where $h^2$ and $h^3$ denote factors of the form $h\ast h$ and $h\ast h\ast h$, respectively.
\end{lemma}
\begin{proof} To compute this commutator, we fix $t_0 \in [0, \Omega]$
and ``freeze'' the connection at time $t_0$, using $\overline{\nabla}$ to 
denote the Levi-Civita connection of the metric $\bar{g} \dfn g(t_0)$.  Then we obtain a tensorial
expression for the difference of the second covariant derivatives applied to our tensor $T$.
Namely, we have
\[
  \Gamma_{ij}^k -\overline{\Gamma}_{ij}^k 
= \frac{1}{2}h^{mk}\left(\overline{\nabla}_i g_{jm} + \overline{\nabla}_j g_{im} - \overline{\nabla}_m g_{ij}\right),
\]
so
\[
 \nabla - \delb \lt 3h\ast\delb g, \quad \nabla T \lt \overline{\nabla} T + 3\rho h\ast \overline{\nabla}g\ast T,  
\]
and, consequently,
\begin{align*}
\begin{split}
 \nabla\nabla T &\lt \delb \nabla T + 3(\rho +1)h\ast\delb g\ast \nabla T\\
  &\lt \delb\left(\delb T + 3\rho h\ast\delb g\ast T\right) +3(\rho+1)h\ast \delb g \ast \left(\delb T + 3\rho h\ast\delb g\ast T\right)\\
&\lt \overline{\nabla}\overline{\nabla}T 
  + 3\rho(3\rho+4)h \ast h\ast \overline{\nabla}g \ast \overline{\nabla} g \ast T+3\rho h\ast \overline{\nabla}\overline{\nabla}g \ast T \\
  &\phantom{\lt}
  + 3(2\rho + 1) h\ast \overline{\nabla}g \ast \overline{\nabla}T,
\end{split}
\end{align*}
as $\delb h \lt h\ast h\ast \delb g$.
Defining the operator $\overline{\Box}\dfn h^{ij}\delb_i\delb_j$,
we therefore have
\begin{align*}
 \begin{split}
 \Delta T &\lt \overline{\Box}T 
  + 3\rho(3\rho+4) h^3\ast \overline{\nabla}g \ast \overline{\nabla} g \ast T\\
  &\qquad\phantom{\lt}+3\rho h^2\ast \overline{\nabla}\overline{\nabla}g \ast T 
  + 3(2\rho + 1) h^2\ast \overline{\nabla}g \ast \overline{\nabla}T.
\end{split}
\end{align*}
Since the connection $\overline{\nabla}$ is independent of time, 
differentiating the expression on the right-hand side of the above equation is merely a matter
of applying the Leibniz rule, and doing so yields
\begin{align*}
\begin{split}
\dt{l}\Delta T &\lt \overline{\Box}\T{0}{l} + \sum_{|\gamma| = l-1}\binom{l}{\gamma}\h{0}{\gamma_1+1}\ast\delb\delb\T{0}{\gamma_2}\\
 &\phantom{\lt}
  + 3(3\rho+1)(\rho+1)\sum_{|\gamma|=l}\binom{l}{\gamma} \left(h^3\right)^{(0, \gamma_1)}\ast \overline{\nabla}\g{0}{\gamma_2} 
  \ast \overline{\nabla}\g{0}{\gamma_3} \ast \T{0}{\gamma_4}\\
  &\phantom{\lt}+3\rho \sum_{|\gamma|=l}\binom{l}{\gamma}\left(h^2\right)^{(0, \gamma_1)}\ast 
	  \overline{\nabla}\overline{\nabla}\g{0}{\gamma_2} \ast \T{0}{\gamma_3}\\
  &\phantom{\lt}
  + 3(2\rho + 1) \sum_{|\gamma|=l}\binom{l}{\gamma} \left(h^2\right)^{(0, \gamma_1)}\ast \overline{\nabla}\g{0}{\gamma_2} \ast \overline{\nabla}\T{0}{\gamma_3}.
\end{split}
\end{align*}

Now we evaluate this expression at $t= t_0$.  We have $\nabla = \delb$, $\Delta =\overline{\Box}$, $\delb g = 0$, and
also that, for any $p\in \NN_0$ and $q \in\NN$, 
\[
\delb^{(p)}\g{0}{q} = \ -2\nabla^{(p)}\dt{q-1}\Rc \lt 2n \R{p}{q-1}.
\]
So we have
\begin{align*}
 \begin{split}
\left[\dt{l}, \Delta\right] T&\lt C(\rho, n)\bigg\{\sum_{\stackrel{|\gamma|=l}{\gamma_2 > 0}}\binom{l}{\gamma} \left(h^2\right)^{(0, \gamma_1)}\ast \R{1}{\gamma_2-1} \ast \T{1}{\gamma_3}\\
 &\phantom{\lt C(\rho, n)\bigg\{}
  + \sum_{\stackrel{|\gamma|=l}{\gamma_2, \gamma_3 > 0}}\binom{l}{\gamma} \left(h^3\right)^{(0, \gamma_1)}\ast \R{1}{\gamma_2-1} 
  \ast \R{1}{\gamma_3-1} \ast \T{0}{\gamma_4}\\
  &\phantom{\lt C(\rho, n)\bigg\{}+\sum_{\stackrel{|\gamma|=l}{\gamma_2 > 0}}\binom{l}{\gamma}\left(h^2\right)^{(0, \gamma_1)}\ast 
	  \R{2}{\gamma_2-1} \ast \T{0}{\gamma_3}\\
  &\phantom{\lt C(\rho, n)\bigg\{}
  +\sum_{|\gamma| = l-1}\binom{l}{\gamma}\h{0}{\gamma_1+1}\ast\T{2}{\gamma_2}
\bigg\},
\end{split}
\end{align*}
which is the desired formula.
\end{proof}

Specializing to the case $T = R$, and applying a $k$-fold covariant differentiation to the result of Lemma \ref{lem:dtldeltacomm},
we obtain the following expression for   $W_{k, l} = \nabla^{(k)}[\dt{l}, \Delta]R$.
\begin{lemma}\label{lem:wkl} There exists a constant $C_6 = C_6(n)$ such that, for any
$k$, $l\in \NN_0$, we have $W_{k, l} \lt C_6(W_{k, l}^1 + W_{k, l}^2 + W_{k, l}^3 + W_{k, l}^4)$ 
where
\begin{align}\label{eq:wkl1}
W_{k, l}^1 &\dfn \sum_{\stackrel{|\beta| =k, |\gamma| = l}{\gamma_2 \geq 1}}\binom{k}{\beta}\binom{l}{\gamma}\left(h^2\right)^{(\beta_1, \gamma_1)}\ast \R{\beta_2 + 1}{\gamma_2-1} \ast 
      \R{\beta_3 + 1}{\gamma_3},\\
\label{eq:wkl2}
W_{k, l}^2  &\dfn\sum_{\stackrel{|\beta|= k, |\gamma|=l}{\gamma_2, \gamma_3 \geq 1}}
	  \binom{k}{\beta}\binom{l}{\gamma} \left(h^3\right)^{(\beta_1, \gamma_1)}\ast \R{\beta_2+1}{\gamma_2-1} 
      \ast\R{\beta_3+ 1}{\gamma_3-1}
		\ast \R{\beta_4}{\gamma_4}, \\
\label{eq:wkl3}
 W_{k, l}^3 &\dfn \sum_{\stackrel{|\beta|= k, |\gamma| = l}{\gamma_2\geq 1}}\binom{k}{\beta}\binom{l}{\gamma}
      \left(h^2\right)^{(\beta_1, \gamma_1)}
	  \ast \R{\beta_2+2}{\gamma_2-1}\ast \R{\beta_3}{\gamma_3}, \quad\mbox{and}\\
\label{eq:wkl4}
W_{k, l}^4  &\dfn \sum_{|\beta|= k, |\gamma|=l-1}
	    \binom{k}{\beta}\binom{l}{\gamma}\h{\beta_1}{\gamma_1+1}\ast \R{\beta_2+2}{\gamma_2}.
\end{align}
\end{lemma}

We now proceed to estimate $\mc{W}^i_{k, l} \dfn a_{k, l}^2\theta^{k+2l}|W_{k, l}^i|^2$ for $i =1, 2, 3, 4$;
in what follows, we may assume $l \geq 1$ as $W_{k, 0} \equiv 0$.  We will use the temporary shorthand
$\chi_{k, l}^{(i)} \dfn a_{k, l}^2\theta^{k+2l}|(h^i)^{(k, l)}|^2$, for $i = 2$, $3$.

\subsubsection{Estimate of $\mc{W}_{k, l}^1$}
 We begin with equation \eqref{eq:wkl1},
and argue as in Lemma \ref{lem:prodsum}.  First, we have
\begin{align*}
\begin{split}
\sqrt{\mc{W}^1_{k, l}} &\leq
\sum_{\stackrel{|\beta| =k, |\gamma| = l}{\gamma_2 \geq 1}}\frac{[k]_2 [l]_2[\beta_2-1][\beta_3-1]}{[\beta]_2[\gamma]_2[\gamma_2-2]}
	\sqrt{\chi^{(2)}_{\beta_1, \gamma_1} \phi_{\beta_2 + 1, \gamma_2-1}
	  \phi_{\beta_3 + 1, \gamma_3}}.
\end{split}
\end{align*}
Since $[\beta_2 - 1] \leq [k-1]$, and $[\beta_3- 1]\phi^{1/2}_{\beta_3+1, \gamma_3} = \theta^{1/2}\psi^{1/2}_{\beta_3+1, \gamma_3}$, this becomes
\begin{align*}
\begin{split}
\sqrt{\mc{W}^1_{k, l}}
&\leq \theta^{1/2}[k-1]\sum_{\stackrel{|\beta| =k, |\gamma| = l}{\gamma_2 \geq 1}}\frac{[k]_2 [l]_2}{[\beta]_2[\gamma]_2}
	\sqrt{\chi^{(2)}_{\beta_1, \gamma_1} \phi_{\beta_2 + 1, \gamma_2-1}
	  \psi_{\beta_3 + 1, \gamma_3}}.
\end{split}
\end{align*}
So, by Lemma \ref{lem:comb1}, we have
\begin{align*}
\mc{W}^1_{k, l}
	      &\leq C\theta[k-1]^2\left(\sum_{\stackrel{|\beta| =k, |\gamma| = l}{\gamma_2 \geq 1}}\frac{[k]_4 [l]_4}{[\beta]_4[\gamma]_4}\right)
	\left(\sum_{\stackrel{|\beta| =k, |\gamma| = l}{\gamma_2 \geq 1}}
	\chi^{(2)}_{\beta_1, \gamma_1} \phi_{\beta_2 + 1, \gamma_2-1}\psi_{\beta_3 + 1, \gamma_3}\right)\\
 &\leq C_7 \theta [k-1]^2\sum_{\stackrel{|\beta| =k, |\gamma| = l-1}{\gamma_2 \geq 1}}
	\chi^{(2)}_{\beta_1, \gamma_1} \phi_{\beta_2 + 1, \gamma_2}
	  \psi_{\beta_3 + 1, \gamma_3},
\end{align*}
for some universal constant $C_7$.  Consequently, we obtain
\begin{align}\nonumber
 \sum_{1\leq k + l \leq N}\frac{\mc{W}^1_{k, l}}{[k-1]^2}
	&\leq C_7 \theta \sum_{\stackrel{1\leq k + l \leq N}{|\beta| =k, |\gamma| = l-1}}
	\chi^{(2)}_{\beta_1, \gamma_1} \phi_{\beta_2 + 1, \gamma_2}
	  \psi_{\beta_3 + 1, \gamma_3}\\
\nonumber
      &\leq C_7\theta \left(\sum_{m=0}^{N-1}\sum_{k=0}^{m}\chi^{(2)}_{k, m-k}\right)
		   \left(\sum_{m=1}^{N}\sum_{k=1}^{m}\phi_{k, m-k}\right)
		   \left(\sum_{m=1}^{N}\sum_{k=1}^{m}\psi_{k, m-k}\right)\\
\label{eq:mcwkl1}
      &\leq C_7\theta\Phi_N\Psi_{N-1}\left(\sum_{m=0}^N\sum_{k=0}^{m}\chi^{(2)}_{k, m-k}\right).
\end{align}

\subsubsection{Estimate of $\mc{W}_{k, l}^2$}
Next, starting from  \eqref{eq:wkl2} and arguing as for the previous term, we find initially that
\begin{align*}
\begin{split}
\sqrt{\mc{W}_{k, l}^2}  &\leq\sum_{\stackrel{|\beta|= k, |\gamma|=l}{\gamma_2, \gamma_3 \geq 1}}
	  \frac{\theta[k]_2[l]_2[\beta_2-1][\beta_3-1]}{[\beta]_2[\gamma]_2[\gamma_2-2][\gamma_3-2]}\sqrt{\chi^{(3)}_{\beta_1, \gamma_1}\phi_{\beta_2+1, \gamma_2-1} 
      \phi_{\beta_3+ 1, \gamma_3-1}\phi_{\beta_4, \gamma_4}}
\end{split}\\
&\leq [k-1]\theta^{3/2}\sum_{\stackrel{|\beta|= k, |\gamma|=l}{\gamma_2, \gamma_3 \geq 1}}
	  \frac{[k]_2[l]_2}{[\beta]_2[\gamma]_2}\sqrt{\chi^{(3)}_{\beta_1, \gamma_1}\phi_{\beta_2+1, \gamma_2-1} 
      \psi_{\beta_3+ 1, \gamma_3-1}\phi_{\beta_4, \gamma_4}}.
\end{align*}
Then, using Lemma \ref{lem:comb1} again, we have
\[
 \sum_{\stackrel{|\beta|= k, |\gamma|=l}{\gamma_2, \gamma_3 \geq 1}}
	  \left(\frac{[k]_2[l]_2}{[\beta]_2[\gamma]_2}\right)^2 
  \leq C\sum_{|\beta|= k, |\gamma|=l}\frac{[k]_4[l]_4}{[\beta]_4[\gamma]_4} \leq C_8,
\]
for some universal $C_8$, and so
\begin{align*}
 \mc{W}_{k, l}^2 &\leq C_8[k-1]^2\theta^3\sum_{|\beta|= k, |\gamma|=l-2}
	 \chi^{(3)}_{\beta_1, \gamma_1}\phi_{\beta_2+1, \gamma_2} 
      \psi_{\beta_3+ 1, \gamma_3}\phi_{\beta_4, \gamma_4}.
\end{align*}
 Summing this inequality over $k$ and $l$ then yields
\begin{align}\nonumber
 \sum_{1\leq k + l \leq N}\frac{\mc{W}^2_{k, l}}{[k-1]^2} 
    &\leq \sum_{\stackrel{1\leq k + l \leq N}{|\beta|= k, |\gamma|=l-2}}
	 C_8\theta^3 \chi^{(3)}_{\beta_1, \gamma_1}\phi_{\beta_2+1, \gamma_2} 
      \psi_{\beta_3+ 1, \gamma_3}\phi_{\beta_4, \gamma_4}\\
\begin{split}\nonumber  
  &\leq C_8\theta^3 \left(\sum_{m=0}^{N-2}\sum_{k=0}^{m}\chi^{(3)}_{k, m-k}\right)
		   \left(\sum_{m=1}^{N-1}\sum_{k=1}^{m}\phi_{k, m-k}\right)
		   \left(\sum_{m=1}^{N-1}\sum_{k=1}^{m}\psi_{k, m-k}\right)\\
  &\phantom{\leq}\quad\qquad\times\left(\sum_{m=0}^{N-2}\sum_{k=0}^{m}\phi_{k, m-k}\right)
\end{split}\\
\label{eq:mcwkl2}
  &\leq C_8\theta^3 \Phi_{N-1}\Phi_{N-2}\Psi_{N-2}\left(\sum_{m=0}^{N-2}\sum_{k=0}^{m}\chi^{(3)}_{k, m-k}\right).
\end{align}
\subsubsection{Estimate of $\mc{W}_{k, l}^3$}
We proceed as for the previous two terms.  Equation \eqref{eq:wkl3} implies 
\begin{align*}
\begin{split}
  \sqrt{\mc{W}_{k, l}^3} &\leq \sum_{\stackrel{|\beta|= k, |\gamma| = l}{\gamma_2\geq 1}}
      \frac{[k]_2[l]_2[\beta_2]_2}{[\beta]_2[\gamma]_2[\gamma_2-1]}
      \sqrt{\chi^{(2)}_{\beta_1, \gamma_1}\phi_{\beta_2+2, \gamma_2-1}\phi_{\beta_3, \gamma_3}}
\end{split}\\
  &\leq \sqrt{\theta}\sum_{\stackrel{|\beta|= k, |\gamma| = l}{\gamma_2\geq 1}}
      \frac{[k]_2[l]_2}{[\beta]_2[\gamma]_2[\beta}
      \sqrt{\chi^{(2)}_{\beta_1, \gamma_1}\psi_{\beta_2+2, \gamma_2-1}\phi_{\beta_3, \gamma_3}}\\
  &\leq [k-1]\sqrt{\theta}\sum_{\stackrel{|\beta|= k, |\gamma| = l}{\gamma_2\geq 1}}
      \frac{[k]_2[l]_2}{[\beta]_2[\gamma]_2}
      \sqrt{\chi^{(2)}_{\beta_1, \gamma_1}\psi_{\beta_2+2, \gamma_2-1}\phi_{\beta_3, \gamma_3}},
\end{align*}
so, as above, we obtain that
\begin{align*}
 \mc{W}_{k, l}^3 &\leq C_9[k-1]^2\theta  \sum_{|\beta|= k, |\gamma| = l-1}
      \chi^{(2)}_{\beta_1, \gamma_1}\psi_{\beta_2+2, \gamma_2}\phi_{\beta_3, \gamma_3},
\end{align*}
for some universal constant $C_9$.  Summing, we then find that
\begin{align}\nonumber
 \sum_{1\leq k + l \leq N}\frac{\mc{W}^3_{k, l}}{[k-1]^2}
  &\leq C_9\theta\sum_{\stackrel{1\leq k + l \leq N}{|\beta|= k, |\gamma| = l-1}}
      \chi^{(2)}_{\beta_1, \gamma_1}\psi_{\beta_2+2, \gamma_2}\phi_{\beta_3, \gamma_3}\\
\nonumber
  &\leq   C_9\theta\left(\sum_{m=0}^{N-1}\sum_{k=0}^m  \chi^{(2)}_{k, m-k}\right)
	  \left(\sum_{m=2}^{N+1}\sum_{k=2}^m \psi_{k, m-k}\right)
	  \left(\sum_{m=0}^{N-1}\sum_{k=0}^m\phi_{k, m-k}\right)\\
\label{eq:mcwkl3}
  &\leq C_9\theta\Phi_{N-1}\Psi_{N}\left(\sum_{m=0}^{N-1}\sum_{k=0}^{m}\chi^{(2)}_{k, m-k}\right).
\end{align}

\subsubsection{Estimate of $\mc{W}_{k, l}^4$} This term is potentially the most troublesome of the four,
but it can be handled in the manner of the previous three with a minor adjustment. Starting from \eqref{eq:wkl4}, 
we have, for $0< t \leq \Omega$,
\begin{align*}
\begin{split}
  \sqrt{\mc{W}_{k, l}^4} &\leq  \theta^{-1}\sum_{\stackrel{|\beta|= k, |\gamma| = l}{\gamma_1\geq 1}}
      \frac{[k]_2[l]_2[\beta_2]_2}{[\beta]_2[\gamma]_2}
      \sqrt{\chi_{\beta_1, \gamma_1}\phi_{\beta_2+2, \gamma_2}}
\end{split}\\
  &\leq \theta^{-1/2}\sum_{\stackrel{|\beta|= k, |\gamma| = l}{\gamma_1\geq 1}}
      \frac{[k]_2[l]_2[\beta_2 - 1]}{[\beta]_2[\gamma]_2}
      \sqrt{\chi_{\beta_1, \gamma_1}\psi_{\beta_2+2, \gamma_2}}\\
 &\leq [k-1]\theta^{-1/2}\sum_{\stackrel{|\beta|= k, |\gamma| = l}{\gamma_1\geq 1}}
      \frac{[k]_2[l]_2}{[\beta]_2[\gamma]_2}
      \sqrt{\chi_{\beta_1, \gamma_1}\psi_{\beta_2+2, \gamma_2}},
\end{align*}
so, by Lemma \ref{lem:comb1} and Cauchy-Schwarz,
\begin{align*}
  \mc{W}^4_{k, l} \leq \frac{C_{10}[k-1]^2}{\theta}\sum_{|\beta|= k, |\gamma| = l-1}
    \chi_{\beta_1, \gamma_1+1}\psi_{\beta_2+2, \gamma_2}
\end{align*}
for some universal constant $C_{10}$.  Summing, we then obtain that 
\begin{align}\nonumber
 \sum_{1\leq k + l \leq N}\frac{\mc{W}^4_{k, l}}{[k-1]^2} &\leq \frac{C_{10}}{\theta} 
    \left(\sum_{m=0}^{N}\sum_{k=0}^{m-1}\chi_{k, m-k}\right)\left(\sum_{m=1}^{N+1}\sum_{k=2}^m\psi_{k, m-k}\right)\\
\label{eq:mcwkl4}
&\leq C_{10}\theta^{-1}(X_N - n) \Psi_N,
\end{align}
where we have used that
\[
    X_N = \sum_{m=0}^{N}\sum_{k=0}^{m}\chi_{k, m-k} = \chi_{0, 0} + \sum_{m=1}^{N}\sum_{k=0}^{m-1}\chi_{k, m-k} = n + \sum_{m=1}^{N}\sum_{k=0}^{m-1}\chi_{k, m-k}.
\]
These latter equalities are valid since, as noted before, $\chi_{0, 0} \equiv n$, and $\chi_{k, 0} \equiv 0$ if $k > 0$.

\subsubsection{Combined estimate for $\mc{W}_{k, l}$}

Now we bring estimates \eqref{eq:mcwkl1}, \eqref{eq:mcwkl2}, \eqref{eq:mcwkl3}, and \eqref{eq:mcwkl4} together.
\begin{proposition}\label{prop:wklest}
  For any $\Omega$ and $M_0 > 0$, there exist constants $C_{11} = C_{11}(n)$ and $L_6 = L_6(n, M_0, \Omega^{*})$
such that if $g(t)$ satisfies the assumptions of Theorem \ref{thm:timeanalyticity}, then, for any $L \geq L_6$
and $N\in \mathbb{N}$, if $\Phi_{N-1} = \Phi_{N-1; L}$ satisfies
\[
  \sup_{M\times[0, \Omega]}\Phi_{N-1} \leq 2M_0^2,
\]
the quantity $\mc{W}_{k, l}$ formed from $\Phi_{N}$ as in \eqref{eq:mcuvwdef} with $\theta = \theta_L$ 
satisfies
\begin{equation}\label{eq:wklest}
 \sum_{1\leq k+l\leq N} \frac{\theta\mc{W}_{k, l}}{[k-1]^2} \leq \left(C_{11}\theta(\Phi^2 + 1) + \frac{1}{2C_2}\right)\Psi_N 
\end{equation}
for any $N\in \mathbb{N}$ on $M\times [0, \Omega]$, where $C_2 = C_2(n)$ is the constant from Proposition 
\ref{prop:prelimphibound}.
\end{proposition}
\begin{proof}
Let the constant $C_6 = C_6(10)$ be as in Lemma \ref{lem:wkl} and $C_{10}= C_{10}(n)$ be as in \eqref{eq:mcwkl4}.
We first apply Lemma \ref{lem:invmetbounds} to obtain $L_5 = L_5(n, 2M_0^2, \Omega^*)$ 
sufficiently large to ensure that
$X_{N} = X_{N; L}$ satisfies
\begin{equation}\label{eq:xnbound}
    X_N\leq n\left(1 + \frac{1}{2nC_2C_{6}C_{10}}\right) 
\end{equation}
on $M\times [0, \Omega]$,
and take our $L_6 \geq \max\{L_5, \Omega^*\}$, henceforth assuming $L \geq L_6$. Note that, according to \eqref{eq:mcwkl4}, we then have
\begin{equation}\label{eq:wkl4new}
 \sum_{1\leq k + l \leq N}\frac{\theta\mc{W}^4_{k, l}}{[k-1]^2} \leq C_{10}(X_N - n) \Psi_N \leq \frac{1}{2C_2C_6}\Psi_N. 
\end{equation}
The bound \eqref{eq:xnbound}, together with Lemma \ref{lem:prodsum}, implies that the sums of $\chi^{(i)}_{k, l}$,
$i = 2, 3$, namely
\[
 \sum_{m=0}^{N}\sum_{k=0}^{m}\chi^{(i)}_{k, m-k},
\]
which appear as factors in equations  \eqref{eq:mcwkl1}, \eqref{eq:mcwkl2}, and \eqref{eq:mcwkl3},
are also bounded above by a constant, $C^{\prime}$, depending only on $n$.

Returning, then, to the decomposition of $\mc{W}_{k, l}$ defined in Lemma  \ref{lem:wkl},
and summing the inequalities \eqref{eq:mcwkl1}, \eqref{eq:mcwkl2}, \eqref{eq:mcwkl3}, and \eqref{eq:wkl4new},
we obtain
\begin{align*}
 &\sum_{0\leq k+l\leq N}  \frac{\theta\mc{W}_{k, l}}{[k-1]^2} \leq \sum_{0\leq k+l\leq N}\frac{C_6\theta}{[k-1]^2}
	  \left(\mc{W}^1_{k, l}+ \mc{W}^2_{k, l}+ \mc{W}^3_{k, l} + \mc{W}^4_{k, l}\right)\\
    &\qquad\leq C_6C^{\prime}\theta^2\left(C_7\Phi_N\Psi_{N-1} + C_8\theta^2 \Phi_{N-1}\Phi_{N-2}\Psi_{N-2} + 
    C_9\Phi_{N-1}\Psi_{N}\right) + \frac{1}{2C_2}\Psi_N.\\
    &\qquad\leq C_{11}\theta^2(1+ \Phi^2_N)\Psi_{N} + \frac{1}{2C_2}\Psi_N.
\end{align*}
provided $C_{11} = C_{11}(n)$ is taken sufficiently large.
Here we have used that $\theta \leq 1$ if $L\geq L_6$ and that $\Phi_{N^{\prime}} \leq \Phi_{N^{\prime\prime}}$ and $\Psi_{N^{\prime}}\leq \Psi_{N^{\prime\prime}}$
if $N^{\prime}\leq N^{\prime\prime}$.
\end{proof}

\section{Proof of Proposition \ref{prop:phiev}}

Now we are ready to prove Proposition \ref{prop:phiev}, and hence Theorem \ref{thm:phibound}, in view of the argument
in Section \ref{sec:proofthmphibd}.
\begin{proof}
  Let $C_2$, $C_{11}$, $L_3$ and $L_6$ be the constants guaranteed by Propositions \ref{prop:prelimphibound} and \ref{prop:wklest}, and assume 
  $L \geq L_2 \dfn \max\{L_3, L_6\}$ initially.
  Then, by \eqref{eq:prelimphibound}, we have
\begin{equation}\label{eq:prelim}
 \left(\pdt-\Delta\right)\Phi_N \leq -\frac{3}{2}\Psi_{N} + 
    \sum_{1\leq k + l \leq N}\frac{C_2\theta\mc{S}_{k, l}}{[k-1]^2}
      + C_2M_0\phi_{0, 0}.
\end{equation}
on $M\times [0, \Omega]$. 
Now, since $\mc{S}_{k, l} \dfn \mc{U}_{k, l} + \mc{V}_{k, l} + \mc{W}_{k, l}$, from equations \eqref{eq:uklest},
\eqref{eq:vklest} and \eqref{eq:wklest}, it follows that
\begin{align*}
  &\sum_{m=1}^N\sum_{k=0}^{m}\frac{C_2\theta\mc{S}_{k, m-k}}{[k-1]^2} 
    = \sum_{m=1}^N\sum_{k=0}^{m}\frac{C_2\theta}{[k-1]^2}\left(\mc{U}_{k, m-k} + \mc{V}_{k, m-k} + \mc{W}_{k, m-k}\right)\\
    &\qquad\qquad\leq C_2 C_4 \theta X_N \Phi_N^2 + C_2C_5\theta^2\Phi_{N}\Psi_N + C_2\left(C_{11}\theta^2(\Phi^2 + 1) 
      + \frac{1}{2C_2}\right)\Psi_N. 
\end{align*}
In view of our choice of $L_6$ from Proposition \ref{prop:wklest} (made via Lemma \ref{lem:invmetbounds}), 
we have $X_N \leq C$ for some $C =C(n)$,
Thus, for sufficiently large $C^{\prime}$, we'll have
\[
 \sum_{m=1}^N\sum_{k=0}^{m}\frac{ C_2\theta\mc{S}_{k, m-k}}{[k-1]^2} \leq C^{\prime}\theta\Phi_N^2 + C^{\prime}\theta^2(\Phi_N^2 + 1)\Psi_N 
+ \frac{1}{2}\Psi_N.
\]
Combined with \eqref{eq:prelim}, this yields
\begin{align*}
 \left(\pdt-\Delta\right)\Phi_N &\leq -\Psi_N + C^{\prime}\theta \Phi_N^2 + C^{\prime}\theta^2(\Phi_N^2 + 1)\Psi_N  + C_2M_0\phi_{0, 0}\\
	&\leq -\Psi_N(1 - C^{\prime}\theta^2(\Phi^2_N + 1)) +C^{\prime\prime}(\theta + 1)(M_0^2+1)(\Phi_N^2 + 1).
\end{align*}
Since $L \geq L_2$, we have $\theta \leq 1$ in particular, so 
\[
 \left(\pdt-\Delta\right)\Phi_N \leq -\Psi_N(1 - C_1\theta^2(\Phi^2_N + 1)) +C_1(M_0^2+1)(\Phi_N^2 + 1),
\]
for a sufficiently large $C_1 = C_1(n)$, for all $L \geq L_2$.
\end{proof}

\section{Local space-time analyticity of the metric in normal coordinates.}

We now adopt the setting of Theorem \ref{thm:spacetimeanalyticity}.  Our aim is to show that, on a sufficiently small
neighborhood of $(x_0, t_0) \in M\times (0, \Omega)$, there are constants $P$ and $Q$ such that 
 the representation $g_{ij}$ of the metric in the coordinates $x$
from the statement of that theorem satisfies
\begin{equation*}\label{eq:metcoordest}
  \left|\frac{\partial^{|\alpha|}\partial^{l}}{\partial x^{\alpha}\partial t^l}g_{ij}\right|\leq PQ^{|\alpha|/2+l}(|\alpha|+l)!
\end{equation*}
for all $l \in \NN_0$ and multi-indices $\alpha$.  On the coordinate neighborhood $U$ where we will be restricting
our attention, we have a natural reference euclidean metric $(g_E)_{ij} = \delta_{ij}$ and
we may 
regard partial differentiation as covariant differentiation relative to the flat connection
associated to $g_{E}$.
With this understanding, we may regard $\partial^{(k)}g$ on this neighborhood as a \emph{tensor} and seek to establish the analogous bounds
\begin{equation}\label{eq:metcoordest2}
 \left|\dt{l}\partial^{(k)}g\right|_{g_E} \leq PQ^{k/2 + l}(k + l)!
\end{equation}
on its time derivatives. To obtain these bounds we will need to convert between estimates involving the three connections
relevant to our problem: $\nabla = \nabla_{g(t)}$, $\delb = \nabla_{g(t_0)}$, and $\partial$, and we will prove
two general propositions in Section \ref{sec:diffconn} to aid us in this effort before embarking on the proof of 
Theorem \ref{thm:spacetimeanalyticity} in Section \ref{sec:spacetimeproof}.

 Before doing so, we point out that, since the conclusion of Theorem \ref{thm:timeanalyticity} is tensorial, 
 we already have the necessary estimates on the pure time derivatives of the metric.  In fact, for some $K_0$ and $L_0$,
\begin{equation}\label{eq:puretimederivs}
      \left|\dt{l} g_{ij}(x, t)\right|_{g_{E}} \leq K_0 \left(\frac{L_0}{t_0 - \eta}\right)^{l}l!
\end{equation}
on $W\times (t_0 -\eta, t_0+\eta)$ for \emph{any} fixed  coordinates on a precompact open set $W$ with $\overline{W}\subset U$
and sufficiently small $\eta > 0$.  (The restriction to $W$ is made so that $g(t)$ and $g_E$
will be uniformly equivalent on $W\times [t_0 - \eta, t_0 + \eta]$.) 
Accordingly, we could appeal now to a classical theorem of Browder \cite{Browder} 
(cf. Theorem 4.33, \cite{KrantzParks}) which provides a condition under which a smooth function
that is separately real-analytic in each of its variables will be fully analytic on its domain.
On account of the uniformity of the estimate \eqref{eq:puretimederivs} in $x$ and $t$, to prove 
Theorem \ref{thm:spacetimeanalyticity},
it would be enough, by Browder's theorem, to obtain a further uniform estimate on the pure spatial (coordinate) derivatives 
$\partial^{(k)}g$ of $g$ over $W\times [t_0-\eta, t_0+\eta]$ (in fact, only on the pure iterated derivatives
$\partial^k_{x^1} g$, $\partial^k_{x^2}g$, etc.).  Given the estimates we have already proven
in Theorem \ref{thm:mainest}, however, such a reduction will spare us little additional effort and instead we continue
to estimate all of the derivatives directly in order that our proof might be fully self-contained.

\subsection{Comparing the difference of connections}\label{sec:diffconn}

In this section we temporarily depart from the setting of the Ricci flow
to establish two general quantitative results that measure the effect that a change
of connection has on derivative estimates of the form \eqref{eq:mainestvar}.  The first
compares two fixed connections whose difference is known to satisfy
some derivative estimates of the same general form as the tensor being estimated.

\begin{proposition}\label{prop:connchange} Suppose that $g$ and $\bar{g}$ are two Riemannian metrics on a manifold $M$ with
Levi-Civita connections $\nabla$ and $\delb$, respectively.  Let $G \dfn \nabla - \delb$, i.e., 
$G_{ij}^k = \Gamma_{ij}^k - \Gamb_{ij}^k$. Let  $U\subset M$ be a precompact open set and $T$ a smooth tensor
of rank $\rho$ on $U$.  Given constants $P$, $\tilde{P}$, and $Q$ such that
\begin{equation}\label{eq:test}
      \sup_{x\in U}|\nabla^{(k)}T(x)| \leq PQ^k(k-2)!
\end{equation}
for all $0 \leq k \leq m$, and \emph{either}
\begin{equation}\label{eq:connest}
  \sup_{x\in U}|\nabla^{(k)}G(x)| \leq \tilde{P}Q^k(k-2)!\quad\mbox{ or }\quad  \sup_{x\in U}|\delb^{(k)}G(x)| \leq \tilde{P}Q^k(k-2)!
\end{equation}
for all $0\leq k \leq m$, there exists a constant $S$ depending only on $n$, $\rho$, $\tilde{P}$, and $Q$
such that
\begin{equation}\label{eq:newtest}
      \sup_{x\in U}|\delb^{(k)}T(x)| \leq PS^{k}(k-2)!.
\end{equation}
for all $0\leq k \leq m$.  Here $|\cdot| = |\cdot|_{g}$, however, a similar estimate holds instead with
$|\cdot|_{\bar{g}}$ in view of the uniform equivalence of $g$ and $\bar{g}$ on $U$.
\end{proposition}
\begin{remark}
 A key requirement for our specific application of this result in the proof of Theorem \ref{thm:spacetimeanalyticity} is that
 \eqref{eq:test} and \eqref{eq:newtest} hold with the \emph{same} $P$ and, especially, that $Q$ is independent of $P$.
 We will be taking as our $T$ various time-derivatives $\dt{l}R$, for which $P$ will depend on $l$, and if this requirement is met,
 the estimates
 for $\delb^{(k)}\dt{l}R$ supplied by this proposition will retain dependencies on $k$ and $l$ of the appropriate form.
\end{remark}

\begin{proof} We begin by assuming that the first alternative in \eqref{eq:connest} is satisfied. 
We will use a double induction argument to prove the following statement.

\noindent{\it Claim: There exists $S = S(n, \rho, \tilde{P}, Q)$ such that
\begin{equation}\label{eq:mixedconnest}
   \sup_{x\in U}|\nabla^{(k-l)}\delb^{(l)} T(x)| \leq P\hat{Q}^{k-l}S^{l}(k-2)!
\end{equation}
 for all $0\leq l\leq k\leq m$, where $\hat{Q} = 3Q$.}  

We will wait to specify $S$ until the end of the argument; for now we just regard it as a large parameter satisfying
 $S\geq Q$.
First note that \eqref{eq:mixedconnest} follows trivially in the case $k=0$ since we have $|T|\leq P$ by \eqref{eq:test}.
Suppose now that $S$ has been found so that, for some $r\leq m$, \eqref{eq:mixedconnest} holds for all $0 \leq l \leq k < r$.
We argue by induction on $l$ to establish the case $k=r$.  Again, the case $l = 0$ is an immediate consequence of \eqref{eq:test}
since $\hat{Q}\geq Q$.
So suppose that \eqref{eq:mixedconnest} holds for all $0\leq l < s$ for some $s\leq k= r$.  We consider the case $l =s\geq 1$.
Observe that
\begin{align}\begin{split}\label{eq:diffschem}
 \nabla^{(r-s)}\delb^{(s)}T &= \nabla^{(r-s+1)}\delb^{(s-1)}T + \nabla^{(r-s)}\left((\delb -\nabla)\delb^{(s-1)}T\right)\\
			   &\lt \nabla^{(r-s+1)}\delb^{(s-1)}T + (s-1 +\rho)\nabla^{(r-s)}\left(G\ast\delb^{(s-1)}T\right)\\
			   &\lt \nabla^{(r-s+1)}\delb^{(s-1)}T\\
      &\phantom{\lt}
      + (s - 1 +\rho)\sum_{q=0}^{r-s}\binom{r-s}{q}\left(\nabla^{(q)}G\ast\nabla^{(r-s- q)}\delb^{(s-1)}T\right),
\end{split}
\end{align}
and so, using the induction hypothesis and \eqref{eq:connest}, we have
\begin{align}
\begin{split}\nonumber
  \left|\nabla^{(r-s)}\delb^{(s)}T \right| &\leq P\hat{Q}^{r-s+1} S^{s-1}(r-2)!\\
    &\phantom{\leq}+ (s -1 + \rho)\sum_{q=0}^{r-s}\binom{r-s}{q}P\tilde{P}Q^q\hat{Q}^{r-s-q}S^{s-1}
      (q-2)!(r-q -3)!
\end{split}\\
\begin{split}\label{eq:indhypprelim}
  &\leq P\hat{Q}^{r-s} S^{s}(r-2)!\left\{\frac{\hat{Q}}{S} + \frac{\tilde{P}}{S}\sum_{q=0}^{r-s} \frac{F(q, r, s; \rho)}{3^{q}}\right\},
\end{split}
\end{align}
where
\[
  F(q, r, s; \rho)\dfn \frac{s-1+\rho}{(r-2)!}\binom{r-s}{q}(q-2)!(r- q -3)! = \frac{(s-1+\rho)[r-s]_q}{[q]_2[r-2][r-3]_q}.  
\]
Now, since $1 \leq s \leq r$, we have $(s-1 +\rho)\leq \rho[r]$, and also
\[
      \frac{[r-s]_q}{[r-3]_q} \leq \frac{[r-1]_q}{[r-3]_q} 
= \frac{[r-1]}{[r-3]}\cdot \frac{[r-2]}{[r-4]} \cdots \frac{[r-q]}{[r-2-q]}
\leq 3^q,
\]
and hence
\begin{align*}
    F(q, L, N; \rho) \leq \frac{\rho [r]3^q}{[q]_2[r-2]}\leq \frac{2\rho3^q}{[q]_2}.
\end{align*}
Thus, we may return to \eqref{eq:indhypprelim} to obtain
\begin{align*}
 \frac{\left|\nabla^{(r-s)}\delb^{(s)}T \right|}{ P\hat{Q}^{r-s} S^{s}(r-2)!}
  &\leq \left(\frac{\hat{Q}}{S} + \frac{2\rho \tilde{P}}{S}\sum_{q=0}^{r-s}\frac{1}{[q]_2}\right)\\
  &<  \frac{3Q}{S} + \frac{C\tilde{P}\rho}{S}.
\end{align*}
for some universal constant $C$. Provided $S$ is taken larger still to ensure that $S > \max\{6Q, 4C\tilde{P}\rho\}$, the right-hand side
of the inequality
will be less than $1$.  This implies the desired estimate for the case $l=s$,
and, by induction on $l$, proves the desired estimate for the case $k=r$.  Thus, by induction on $k$, the claim is proven
for all $0\leq l \leq k \leq m$.

The argument for the case of the second alternative in \eqref{eq:connest} is very similar to the one we have given
above, and we will only give a sketch.
The idea is to work ``inside out'' relative to the above argument and prove instead the following statement.

\noindent{\it Claim: There exists $S = S(n, \rho, \tilde{P}, Q)$ such that
\begin{equation*}
   \sup_{x\in U}|\delb^{(l)}\nabla^{(k-l)} T(x)| \leq PS^{l}Q^{k-l}(k-2)!
\end{equation*}
 for all $0\leq l\leq k\leq m$.} 

This can be proven by another double induction argument on $k$ and $l$;
the analog of \eqref{eq:diffschem} in this case is
\[
    \delb^{(s)}\nabla^{(r-s)} T \lt \delb^{(s-1)}\nabla^{(r-s+1)}T + (r-s + \rho) \delb^{(s-1)}(G\ast\nabla^{(r-s)}T).
\]
Since the covariant derivatives landing on the difference of connections $G$ will now be taken with respect to the connection
$\delb$, one can use the rightmost assumption from \eqref{eq:connest} together with \eqref{eq:test}
and proceed as before.
\end{proof}

The second result provides an estimate on the derivatives of the 
difference of the Levi-Civita connections of a smooth family of metrics
$g(t)$ at two different times, when the time-derivative of this family is known
to satisfy uniform bounds on its spatial derivatives akin to \eqref{eq:mainestvar}.

\begin{proposition}\label{prop:timeconndiff}
Suppose that $g(t)$ is a smooth family of metrics defined on some open set $U\subset M$ for $t\in [0, \Omega]$ and $V\subset U$ is any precompact open set
with closure contained in $U$. Define $\bar{g} \dfn g(0)$, let $\nabla$ and $\delb$ denote the Levi-Civita connections of $g$ and $\bar{g}$, respectively, and 
define $G \dfn \nabla -\delb$.
 If there are constants $P$ and $Q$ such that
\begin{equation}\label{eq:best}
      \sup_{U\times [0, \Omega]}|\nabla^{(k)}b|\leq P Q^{k/2}(k-2)!
\end{equation}
for all $k \in \NN_0$, where $b(t) \in C^{\infty}(T_2(M))$ is defined by $\dt{} g = -2b$, then there exist constants 
$\epsilon\in (0, \Omega)$, $\tilde{P}$, and $\tilde{Q}$ depending only on $n$, $P$, $Q$, and $\Omega$ such that
\begin{equation}\label{eq:conder}
  \sup_{V\times [0, \epsilon]}\left|\nabla^{(k)}G\right|\leq \tilde{P}\tilde{Q}^{k/2}(k-2)!
\end{equation}
for all $k\in \NN_0$.
\end{proposition}
\begin{proof} Observe first that $\nabla^{(k)}G(x, 0)\equiv 0$ for all $k$, and recall the identity
\begin{equation}\label{eq:gder}
  \pdt G_{ij}^k = g^{mk}\left\{\nabla_m b_{ij} -\nabla_ib_{jm} -\nabla_{j}b_{im}\right\}.
\end{equation}
For the time being, let $0<\epsilon< \Omega$ denote a small constant -- we will later specify
its dependency on the parameters $P$ and $Q$ --
and, for $k$, $N\in \NN_0$, define the functions 
\[
    \eta \dfn \frac{\epsilon}{2Q(t+\epsilon)},\quad f_k \dfn \frac{\eta^k}{((k-1)!)^2}|\nabla^{(k)}G|^2,\quad F_N \dfn \sum_{k=0}^N f_k,\quad\mbox{and}\quad E_N \dfn \sum_{k=1}^{N}[k]f_k.
\]
Then, as $\pdt h^{ij} = 2h^{ik}h^{jl}b_{kl}$, and since $\nabla^{(k)}G$ has rank $k+3$, we have
\begin{align}\label{eq:fkev}
  \pd{f_k}{t} &\leq -\left(\frac{k}{t+\epsilon}- 2(k+3)|b|\right)f_k + \frac{2\eta^k}{((k-1)!)^2}\left\langle\pdt\nabla^{(k)}G, \nabla^{(k)}G\right\rangle.
\end{align}
for each $k$. 

 For $k = 0$, using \eqref{eq:best} and \eqref{eq:gder} with \eqref{eq:fkev}, we have
\begin{align}\label{eq:f0ev}
  \pd{f_0}{t} &\leq (6P +1)f_0 + \left|\pd{G}{t}\right|^2 \leq (6P+1) f_0 + 9PQ^{1/2},
\end{align}
which, since $f_0(x, 0) = 0$, implies $f_0 \leq P^{\prime}$ on $V\times [0, \Omega]$ for some $P^{\prime}= P^{\prime}(P, Q, \Omega^*)$.

For $k > 0$, we start by estimating the inner product in \eqref{eq:fkev} by Cauchy-Schwarz:
\begin{align}\label{eq:innerprodest}
  \frac{2\eta^k}{((k-1)!)^2}\left\langle\pdt\nabla^{(k)}G, \nabla^{(k)}G\right\rangle
    &\leq [k]f_k + \frac{\eta^k}{k!(k-1)!}\left|\pdt\nabla^{(k)}G\right|^2.
\end{align}
Now, with a trivial adjustment to its proof, the formula given in Lemma 13.26 of \cite{RFV2P2} for the 
commutator, $\left[\pdt, \nabla^{(k)}\right]$, relative to a solution of Ricci flow, can be converted into a formula for the same 
commutator relative to our arbitrary smooth family of metrics $g(t)$.  From this formula we obtain that if $T = T(t)$ is a smooth family of tensors of rank $\rho$, then
\begin{equation}\label{eq:bcomm}
    \left|\left[\pdt, \nabla^{(k)}\right] T\right| \leq 3(\rho+1)\sum_{i=1}^k\binom{k+1}{i+1} \left|\nabla^{(i)}b\right|\left|\nabla^{(k-i)}T\right|. 
\end{equation}
Using \eqref{eq:best} and \eqref{eq:gder}, we can then  apply \eqref{eq:bcomm} to $G$ to obtain that 
\begin{align*}
   &\left|\pdt\nabla^{(k)} G\right| \leq \left|\nabla^{(k)}\pd{G}{t}\right| + 12\sum_{i=1}^k\binom{k+1}{i+1} \left|\nabla^{(i)}b\right|\left|\nabla^{(k-i)}G\right|\\
    &\quad\leq 3\left|\nabla^{(k+1)} b\right| + 12\sum_{i=1}^k\binom{k+1}{i+1} \frac{PQ^{i/2}(i-2)!(k-i -1)!}{\eta^{(k-i)/2}}f_{k-i}^{1/2}\\
    &\quad\leq 3PQ^{(k+1)/2}(k-1)! 
  + \frac{12P(k+1)!}{\eta^{k/2}}\sum_{i=1}^k\left(\frac{\epsilon^{i}[k-i] f_{k-i}}{2^i(t+\epsilon)^i[i+1]_3^2[k-i]^{3/2}}\right)^{1/2}.
\end{align*}
Then, using Lemma \ref{lem:comb1} and assuming $0 < t \leq \epsilon$, we can estimate the sum in the second term by
\begin{align*}
  &\sum_{i=1}^k\left(\frac{\epsilon^i[k-i] f_{k-i}}{2^i(t+\epsilon)^i[i+1]_3^2[k-i]^{3/2}}\right)^{1/2}\\
  &\quad\leq C\left(\sum_{i=1}^k\frac{1}{[i]_6[k-i]_3}\right)^{1/2}\left(\sum_{i=1}^k\frac{[k-i]f_{k-i}}{2^i}\right)^{1/2}\\
  &\quad\leq \frac{C}{[k]_3^{1/2}}\left(\sum_{i=1}^k \frac{[k-i]f_{k-i}}{2^i}\right)^{1/2},
\end{align*}
and it follows that
\begin{align*}
 \left|\pdt\nabla^{(k)}G\right|^2  &\leq CP^2Q^{k+1}((k-1)!)^2 + \frac{CP^2((k+1)!)^2}{[k]_3\eta^{k}}
      \sum_{i=1}^k\frac{[k-i]f_{k-i}}{2^i}\\
    &\leq CP^2Q^{k+1}((k-1)!)^2 + \frac{CP^2 k!(k-1)! }{\eta^k}
      \sum_{i=1}^k \frac{[k-i]f_{k-i}}{2^i}.
 \end{align*}
Above, and in the inequalities to follow, $C$ represents a sequence of universal constants. Hence,
\begin{align*}
 \frac{\eta^k}{k!(k-1)!}\left|\pdt\nabla^{(k)}G\right|^2 &\leq \frac{CP^2Q^{k+1}}{[k]}\left(\frac{\epsilon}{2Q(t+\epsilon)}\right)^k  + 
	  CP^2\sum_{i=1}^k \frac{[k-i]f_{k-i}}{2^i}\\
      &\leq \frac{CP^2Q}{2^k} +
	  CP^2\sum_{i=1}^k \frac{[k-i]f_{k-i}}{2^i}.
\end{align*}
Combining this with \eqref{eq:fkev} and \eqref{eq:innerprodest}, we obtain that
\begin{align*}
 \pd{f_k}{t} &\leq -\left(\frac{k}{t+\epsilon}- 2(k+3)P - k\right)f_k + \frac{CP^2Q}{2^k}  + 
	  CP^2\sum_{i=1}^k\frac{[k-i]f_{k-i}}{2^i}\\
	  &\leq -\frac{k}{\epsilon}\left(\frac{1}{2} -\epsilon(8P+1)\right)f_k + \frac{CP^2Q}{2^k}  + 
	  CP^2\sum_{i=1}^k\frac{[k-i]f_{k-i}}{2^i}
\end{align*}
for $k > 0$.  Now,
\[
      \sum_{k=1}^N\sum_{i=1}^k \frac{[k-i]f_{k-i}}{2^i} 
\leq \left(\sum_{i=1}^{N}\frac{1}{2^{i}}\right)\left(\sum_{k=0}^{N-1} [k]f_k\right) \leq E_N +f_0 ,
\]
so, summing over $k$ from $0$ to $N$, and using also \eqref{eq:f0ev},
we find that
\begin{align*}
 \pd{F_N}{t} &= \pd{f_0}{t} + \sum_{k=1}^N\pd{f_k}{t}\\
    &\leq (6P+ CP^2 + 1) f_0 + 9PQ^{1/2} + P^2Q - \epsilon^{-1}\left(\frac{1}{2} -\epsilon\left(8P+1 - CP^2\right)\right)E_N \\
      &\leq P^{\prime\prime} - \epsilon^{-1}\left(\frac{1}{2} -\epsilon P^{\prime\prime}\right)E_N
\end{align*}
for some sufficiently large constant $P^{\prime\prime}$ depending on $P$, $Q$ and $\Omega^*$ (the latter
through the bound $|f_0|\leq P^{\prime}$ found above).  Thus, if $0 < \epsilon < 1/(2P^{\prime\prime})$,
then, since $F_N(x, 0) = 0$, we have $F_N \leq \epsilon P^{\prime\prime}< \frac{1}{2}$ on $V\times [0, \epsilon]$,
and hence for each $0 \leq k \leq N$,
\[
    \left|\nabla^{(k)}G\right| \leq \left(\frac{2Q(t+\epsilon)}{\epsilon}\right)^{k/2}\frac{(k-1)!}{2}\leq 2^{k/2-1}Q^{k/2}(k-1)! \leq \tilde{P}\tilde{Q}^{k/2}(k-2)!
\]
on $V\times [0, \epsilon]$ for suitable $\tilde{P}$ and $\tilde{Q}$.
\end{proof}

\subsection{Proof of Theorem \ref{thm:spacetimeanalyticity}}\label{sec:spacetimeproof}
We now bring Propositions \ref{prop:connchange} and \ref{prop:timeconndiff} together to prove the main result of the section.

\begin{proof}[Proof of Theorem \ref{thm:spacetimeanalyticity}]
 Let $g(t)$ be a complete solution to \eqref{eq:rf} of uniformly bounded curvature on $M\times [0, \Omega]$ and
 fix $t_0\in  (0, \Omega)$. Define $\bar{g} \dfn g(t_0)$, $\epsilon_0 \dfn \min\{t_0, \Omega-t_0\}/2$, and let
$x^i: U\subset M\to \RR{n}$ be smooth coordinates for which the coordinate representation of $\bar{g}$, i.e., $\bar{g}_{ij}$, 
belongs to $C^{\omega}(U)$. In what follows we will regard partial differentiation in these variables 
as covariant differentiation relative
to the associated Euclidean metric $g_E$ on $U$, and view, e.g., $\partial^{(k)}g$  and $\partial^{(k)}\bar{g}$ as tensors on the open set
$U$.

Now fix an arbitrary $x_0\in U$ and select some precompact open set $\tilde{U}$ containing $x_0$ with 
$\tilde{U}\subset U$.  On $\tilde{U}\times [0, \Omega]$, the metrics $g_E$, $\bar{g}$, and $g(t)$
will all be uniformly equivalent. In particular, there exists a constant $\gamma$ such that the inequality
\[
	  \gamma^{-(a+b)}|T|^2_{\bar{g}} \leq |T|_{g(t)}^2 \leq \gamma^{(a+b)}|T|_{\bar{g}}^2, 
\]
holds for any permutation of $g_E$, $\bar{g}$, and $g(t)$ and any $(a, b)$-tensor $T$ on $\tilde{U}\times[0, \Omega]$.
 Since our estimates below will scale with the order
of the spatial derivative,  up to an increase in this scale factor, we may freely 
convert bounds given in terms of one norm to bounds given in terms of another.

From the estimates \eqref{eq:mainest} and the discussion at the beginning of Section \ref{sec:combin}
(in particular, the inequality \eqref{eq:reversefactineq}), it follows that on $\tilde{U}$,
there are constants $P_0$ and $Q_0$ such that
\begin{equation}\label{eq:nablametbounds}
     \sup_{\tilde{U}\times [t_0-\epsilon_0, t_0 + \epsilon_0]}\left|\nabla^{(k)}\dt{l}g\right|
  \leq P_0 Q_0^{k/2 + l}(k-2)!(l-2)!.
\end{equation}
If we define $G \dfn  \nabla - \delb$,
we can apply Proposition \ref{prop:timeconndiff} twice, first to $g_1(t) = g(t_0 - t)$, and then to
 $g_2(t) = g(t_0 + t)$ (both of which are defined on $\tilde{U}\times [0, \epsilon_0]$),  to obtain constants $P_1$ and $Q_1$ and an  
 $\epsilon_1 \in (0, \epsilon_0)$ such that
\begin{equation}\label{eq:gbound}
    \sup_{\tilde{U}\times [t_0-\epsilon_1, t_0 + \epsilon_1]}\left|\nabla^{(k)}G\right| \leq P_1 Q_1^{k/2}(k-2)!
\end{equation}
for any $k\in \NN_0$.
Then we can take \eqref{eq:nablametbounds} and \eqref{eq:gbound} together and apply Proposition \ref{prop:connchange} (with the first
alternative in \eqref{eq:connest}) on each time-slice of $\tilde{U}\times [t_0-\epsilon_1, t_0 + \epsilon_1]$ 
to the tensors $\dt{l}g$ for each $l$.  In these applications, we take the parameters
 $P = P(l)= P_0Q_0^l(l-2)!$, and $\tilde{P} = P_1$, $Q = \max\{Q_0, Q_1\}$, and obtain
$Q_2\geq Q_0$ depending only on $n$, $P_1$, $Q_0$, $Q_1$ (but not on $l$) such that
\[
     \sup_{\tilde{U}\times [t_0-\epsilon_1, t_0+\epsilon_1]} \left|\delb^{(k)}\dt{l} g\right| \leq P(l) Q_2^{k/2}(k-2)!
      \leq P_0 Q_2^{l+k/2}(k-2)!(l-2)!
\]
for all $k$, $l\in \NN_0$. 

Choosing $P_3\geq P_0$ and $Q_3\geq Q_2$ increased by a further factor depending on $\gamma$,
we then may replace the norm in this inequality by $|\cdot|_{\bar{g}}$, obtaining
\begin{equation}\label{eq:metdelbest}
\sup_{\tilde{U}\times [t_0-\epsilon_1, t_0+\epsilon_1]} \left|\delb^{(k)}\dt{l} g\right|_{g_{E}}  
\leq P_3Q_{3}^{k/2+l}(k-2)!(l-2)!
\end{equation}
for all $k$, $l\in \NN_0$.

It remains to translate these estimates into versions relative to the flat connection $\partial$, and it is here that we employ our assumptions on the coordinates $x$.
Note first that, since $\bar{g}_{ij}\in C^{\omega}(U)$, it follows that $\bar{g}^{ij}\in C^{\omega}(U)$, and that the tensor 
\[
    \overline{G}_{ij}^k = (\partial - \delb)_{ij}^k = -\overline{\Gamma}_{ij}^k = \frac{1}{2}\bar{g}^{mk}\left\{\partial_m \bar{g}_{ij} 
    - \partial_{i}\bar{g}_{jm} - \partial_{j}\bar{g}_{im}\right\}, 
\]
as a polynomial expression in $\bar{g}^{-1}$ and $\partial\bar{g}$, likewise belongs to $C^{\omega}(U)$. 
Thus, fixing a precompact neighborhood $V$
of $x_0$ with $\overline{V}\subset \tilde{U}$, we obtain constants $P_4$ and $Q_4$ such that
\begin{equation}\label{eq:gbarbounds}
      \sup_{x\in V}\left|\partial^{(k)}\overline{G}\right|_{\bar{g}} \leq P_4 Q_4^{k/2}(k-2)!
\end{equation}
for all $k\in \NN_0$.  We can then apply Proposition \ref{prop:connchange} again
to $\dt{l} g$ on each time slice of $V\times [t_0-\epsilon, t_0+\epsilon]$, this time with \eqref{eq:metdelbest} and the second
alternative in \eqref{eq:connest} (so $\nabla = \delb$ and $\delb =\partial$, in the notation of its statement).
It produces a constant $Q_5$ (which we may take to exceed $Q_3$), depending only on $n$, $P_4$, $Q_3$, and $Q_4$, such that
\[
    \sup_{V\times [t_0-\epsilon_1, t_0+ \epsilon_1]}
\left|\dt{l}\partial^{(k)} g\right|_{\bar{g}} \leq P_3Q_3^{l}Q_5^{k/2}(k-2)!(l-2)! \leq P_3 Q_5^{k/2+l}(k-2)!(l-2)!
\]
and we thus obtain, for $\epsilon = \epsilon_1$ and suitably larger $P$ and $Q$, that
\[
 \sup_{V\times [t_0-\epsilon, t_0+ \epsilon]}\left|\dt{l}\partial^{(k)}g\right|_{g_{E}} \leq PQ^{k/2+l}(k+l)!.
\]

So, at any $(x, t)\in V\times (t_0 -\epsilon, t_0 + \epsilon)$, for any multi-index $\alpha$ and $l\in \NN_0$,
the representation $g_{ij}$ of the metric $g$ satisfies
\[
 \left|\dt{l}\frac{\partial^{|\alpha|}}{\partial x^{\alpha}}g_{ij}\right| \leq
  \left|\dt{l}\partial^{(|\alpha|)}g\right|_{g_{E}} \leq PQ^{|\alpha|/2+l}(|\alpha|+l)!,
\]
and thus belongs to $C^{\omega}(V\times (t_0-\epsilon, t_0 +\epsilon))$.
\end{proof}

\begin{acknowledgement*}
	The author wishes to thank Professor Gerhard Huisken for useful discussions at the outset of this work.
\end{acknowledgement*}

\end{document}